\documentclass[a4paper]{amsart}

\usepackage[utf8]{inputenc}
\usepackage[british]{babel}
\usepackage{amsfonts,amsmath,amssymb,amsthm,mathtools,esint}
\usepackage[dvipsnames]{xcolor}
\usepackage{tikz}
\usepackage{pgfplots}
\pgfplotsset{compat=1.3}
\usepackage{caption}
\usepackage[justification=centering]{caption}
\usepackage{enumitem}

\usepackage[hidelinks]{hyperref}
\usepackage{cleveref}
\crefname{subsection}{subsection}{subsections}
\numberwithin{equation}{section}

\usepackage[backend=bibtex, giveninits=true, isbn=false, doi=false, url=false, maxnames = 10]{biblatex}
\DeclareFieldFormat[article]{title}{{#1}} 
\DeclareFieldFormat[book]{title}{{#1}}
\DeclareFieldFormat[inproceedings]{title}{{#1}}
\DeclareFieldFormat[incollection]{title}{{#1}}
\DeclareFieldFormat{pages}{{#1}} 
\renewbibmacro{in:}{%
	\ifentrytype{article}
	{}
	{\bibstring{in}%
		\printunit{\intitlepunct}}}
\AtBeginBibliography{\footnotesize}
\newcommand{\vertiii}[1]
{\vert\kern-0.25ex\vert\kern-0.25ex\vert #1 
\vert\kern-0.25ex\vert\kern-0.25ex\vert}
	
\newcommand{\R}{\mathbb{R}}
\renewcommand{\d}[1]{\,\mathrm{d}#1}
\renewcommand{\div}{\mathrm{div}}
\newcommand{\I}{\mathrm{I}}
\newcommand{\tr}{\mathrm{tr}}

\newcommand{\D}{\mathrm{D}}

\newcommand{\Tcal}{\mathcal{T}}
\newcommand{\Fcal}{\mathcal{F}}

\newcommand{\s}{\mathrm{s}}
\newcommand{\pw}{\mathrm{pw}}

\newtheorem{theorem}{Theorem}[section]
\newtheorem{lemma}[theorem]{Lemma}
\newtheorem{corollary}[theorem]{Corollary}
\theoremstyle{remark}
\newtheorem{remark}[theorem]{Remark}

\theoremstyle{definition}

\newtheorem{assumption}{Assumption}

\newcounter{cnst}
\newcommand{\newcnst}{%
	\refstepcounter{cnst}%
	\ensuremath{C_{\thecnst}}}
\newcommand{\cnst}[1]{\ensuremath{C_{\ref{#1}}}}

\begin{document}

\title[HHO for linear elasticity]{Locking-free hybrid high-order method for linear elasticity}
\author[C.~Carstensen, N.~T.~Tran]
      {Carsten Carstensen \and Ngoc Tien Tran}
\thanks{The authors thank Lukas Gehring (Universit\"at Jena) for his valuable comments that lead to the completion of the proof of \Cref{lem:stability-Galerkin}.\\
	\indent The second author received funding from the European Union's Horizon 2020 research and innovation programme (project RandomMultiScales, grant agreement No.~865751).}
\address[C.~Carstensen]{%
         Humboldt-Uni\-ver\-si\-t\"at zu Berlin,
         10117 Berlin, Germany}
\email{cc@math.hu-berlin.de}
\address[N.~T.~Tran]{Universit\"at Augsburg, 86159 Augsburg, Germany}
\email{ngoc1.tran@uni-a.de}
\date{\today}


\begin{abstract}
	The hybrid-high order (HHO) scheme has many successful applications including linear elasticity as the 
	first step towards computational solid mechanics. The striking advantage is the simplicity among other higher-order nonconforming schemes and its geometric flexibility as a polytopal method on the expanse of a parameter-free refined stabilization. 
	This paper utilizes 
	just one reconstruction operator for the linear Green strain and therefore does not rely on a split in deviatoric and spherical behaviour as in the classical HHO discretization.
	The a priori error analysis provides quasi-best approximation with $\lambda$-independent equivalence constants. The reliable and (up to data oscillations) efficient 
	a posteriori error estimates are stabilization-free and $\lambda$-robust.
	The error analysis is carried out on simplicial meshes to allow
	conforming piecewise polynomials finite elements in the kernel of the stabilization terms.
	Numerical benchmarks provide empirical evidence
	for optimal convergence rates of the a posteriori error estimator in some 
	associated adaptive mesh-refining algorithm also in the incompressible limit, where this paper provides corresponding assertions for the Stokes problem.
\end{abstract}
\maketitle
\section{Introduction}

\subsection{Motivation}
The first HHO elasticity model has been analyzed in \cite{DiPietroErn2015} 
with emphasis on general meshes and two different recovery operators for the linear Green strain and for the divergence of the displacements. The a priori error analysis therein provides $\lambda$-robust convergence rates relative to the possibly restricted elliptic regularity on polygons with general boundary conditions of changing
type. The first paper on the HHO in nonlinear elasticity \cite{BottiDiPietroSochala2017} suggests an HHO method with one recovery operator for the linear Green strain tensor.
Optimal convergence rates in \cite{DiPietroErn2015,DiPietroDroniou2020} under unrealistic regularity assumptions on the exact solutions are \emph{not} visible for quasi-uniform triangulations in typical benchmark problems of computational mechanics.

The analysis of this paper departs from $\lambda$-robust quasi-best approximation results for the $L^2$ stress error \cite{CarstensenSchedensack2015,CarstensenGallistlSchedensack2016,LedererStenberg2024} in discrete problems with a smoother \cite{VeeserZanottiII,VeeserZanottiIII,ErnZanotti2020,CN21a} or without (and then up to data oscillations for $L^2$ source terms).
A first benefit is the validity of the error estimates under minimal regularity assumptions.
A second benefit arises in the efficiency of the stabilization -- a stepping stone towards stabilization-free a~posteriori error control \cite{BertrandCarstensenGraessle2021} that is
of interest in the convergence analysis of adaptive mesh-refining algorithms \cite{VeigaCanutoNochettoVaccaVerani2023}.
A duality argument and \cite{CarstensenHeuer2024} lead to $L^2$ estimates for the displacement beyond typical model problems on convex domains with pure Dirichlet boundary conditions \cite{BrennerScott2008,ErnGuermondII2021}.
While HHO methods can be defined on polytopal meshes, the error analysis of this paper exploits the orthogonality of the stress error with conforming test functions. 
Therefore, it is restricted to regular triangulations into simplices.
Other contributions \cite{BirdCoombsGiani2019,MoraRivera2020} on stabilized schemes allow polygonal meshes in the derivation of a~posteriori error estimators but avoid a $\lambda$-robustness analysis.



\subsection{Mathematical model}\label{sec:model}
Let $\Omega \subset \mathbb{R}^n$ with $n = 2,3$ be a bounded polyhedral Lipschitz domain in two or three space dimensions with boundary $\partial \Omega$ that is split into a closed Dirichlet part $\Gamma_\mathrm{D} \subset \partial \Omega$ of positive surface measure and a remaining Neumann part $\Gamma_\mathrm{N} \coloneqq \partial \Omega \setminus \Gamma_\mathrm{D}$.
Given $f \in L^2(\Omega)^n$ and $g \in L^2(\Gamma_{\mathrm{N}})^{n}$, the model problem in linear elasticity of this paper seeks the solution $u \in H^1(\Omega)^n$ to
\begin{align}\label{def:linear-elasticity}
	-\div\,\sigma = f \text{ in } \Omega, \quad\sigma \coloneqq \mathbb{C} \varepsilon(u), \quad u = 0 \text{ on } \Gamma_\mathrm{D}, \quad \sigma \nu = g \text{ on } \Gamma_\mathrm{N}.
\end{align}
The linearized Green strain $\varepsilon(v) \coloneqq \mathrm{sym}(\D v)$ is the symmetric part of the gradient $\D v$ and the isotropic elasticity tensor $\mathbb{C}$ acts as $\mathbb{C} \tau \coloneqq 2\mu \tau + \lambda \mathrm{tr}(\tau) \mathrm{I}_{n \times n}$ for any $\tau \in L^2(\Omega)^{n \times n}$
with the Lam\'e parameters $\lambda \geq 0$ and $\mu > 0$.
We assume that, for $|\Gamma_{\mathrm{N}}| > 0$, the parameters $\lambda$, $\mu$ are piecewise constant such that $\mu$ is bounded away from zero $0 < \mu_0 \leq \mu \leq \mu_1$ by positive numbers $\mu_0 \leq \mu_1$ and, for $|\Gamma_{\mathrm{N}}| = 0$, $\lambda$ and $\mu$ are positive constants.
The weak formulation seeks the solution $u \in V \coloneqq \{v \in H^1(\Omega)^n : v = 0 \text{ on } \Gamma_\mathrm{D}\}$ to
\begin{align}\label{def:continuous-problem}
	\int_\Omega \mathbb{C} \varepsilon (u) : \varepsilon (v) \d{x} = \int_\Omega f \cdot v \d{x} + \int_{\Gamma_\mathrm{N}} g \cdot v \d{s} \quad\text{for all } v \in V.
\end{align}

\subsection{Main results}
The HHO methodology \cite{DiPietroErnLemaire2014,DiPietroErn2015,DiPietroDroniou2020} allows for a reconstruction operator $\mathcal{G} : V_h \to P_k(\Tcal)^{n \times n}$ of the gradient from a discrete ansatz space $V_h$ of $V$ onto the space $P_k(\Tcal)^{n \times n}$ of matrix-valued piecewise polynomials of degree at most $k$.
The approximation $\mathbb{C}\varepsilon_h$ of $\mathbb{C} \varepsilon$ with $\varepsilon_h \coloneqq \mathrm{sym} \,\mathcal{G}$ does not rely on a split in deviatoric and spherical behavior as in \cite{DiPietroErn2015}.
This was proposed in \cite{BottiDiPietroSochala2017} and the
HHO method therein seeks the solution $u_h \in V_h$ to
\begin{align*}
	a_h(u_h,v_h) = F_h(v_h) \quad\text{for any } v_h \in V_h
\end{align*}
with a scalar product $a_h$ in $V_h$ and bounded linear functional $F_h$.
The reconstruction operator $\varepsilon_h$ leads to a straight-forward definition of the discrete stress $\sigma_h \coloneqq \mathbb{C} \varepsilon_h u_h$ with the $L^2$ orthogonality $\sigma - \sigma_h \perp \varepsilon(P_{k+1}(\Tcal)^{n} \cap V)$.
The latter allows for the application of the tr-div-dev lemma in the error analysis for $\lambda$-robust estimates.
We establish, up to data oscillations, the quasi-best approximation
\begin{align}\label{ineq:a-priori}
	(\mu_0/\mu_1)^{1/2}\|\sigma - \sigma_h\| &+ \mu_0^{1/2}\|\I u - u_h\|_{a_h} + \mu_0^{1/2}|u_h|_\s\nonumber\\
	&\leq C_\mathrm{qb}(\|(1 - \Pi_\Tcal^k) \sigma\| + \mathrm{osc}(f,\Tcal) + \mathrm{osc}(g,\Fcal_\mathrm{N}))
\end{align}
under a mild assumption on the geometry,
where $\I : V \to V_h$ is the canonical interpolation in the HHO methodology, $\|\bullet\|_{a_h}$ is the norm induced by $a_h$, $|\bullet|_{\s}$ is the seminorm induced by the stabilization, and $\Pi_{\Tcal}^k$ is the $L^2$ projection onto piecewise polynomials.
Notice that \eqref{ineq:a-priori} holds under minimal regularity assumption on $u$, while
\cite{DiPietroErn2015,DiPietroDroniou2020} requires at least piecewise $H^s$ regularity of the stress $\sigma$ for some $s > 1/2$, which
is not available for boundary conditions of changing type on a straight line.
The proof of \eqref{ineq:a-priori} utilizes the arguments of \cite{ErnZanotti2020,BertrandCarstensenGraessle2021} but additionally requires a quasi-best approximation result for the stabilization in the spirit of \cite{CarstensenZhaiZhang2020}.
Duality techniques allow for $L^2$ error estimates with additional convergence rates depending on the elliptic regularity on polyhedral domains.
If $f \in H^{-1}(\Omega)^n$ and $g \in \widetilde{H}^{-1/2}(\Gamma_{\mathrm{N}})^n$ are nonsmooth, a modified HHO solver with a smoother \cite{VeeserZanottiII,VeeserZanottiIII,ErnZanotti2020,CN21a} is proposed with an oscillation-free version of \eqref{ineq:a-priori}.

The a~posteriori error analysis combines the techniques from the residual and equilibrium error estimation of \cite{BertrandCarstensenGraessle2021}. We establish
\begin{align}\label{ineq:reliability}
	C_\mathrm{rel}^{-1} (\mu_0/\mu_1)\|\sigma - \sigma_h\| &\leq \eta\nonumber\\
	& \leq C_\mathrm{eff}(\mu_1/\mu_0)(\|\sigma - \sigma_h\| + \mathrm{osc}(f,\Tcal) + \mathrm{osc}(g,\Fcal_\mathrm{N})).
\end{align}
with the error estimator
\begin{align}\label{def:eta}
	\eta^2 &\coloneqq \|h_\Tcal(f + \div_\pw \sigma_h)\|^2 + \min_{v \in V} \|\mu(\varepsilon(v) - \varepsilon_h u_h)\|^2\nonumber\\
	&\qquad + \sum_{F \in \Fcal(\Omega)} h_F\|[\sigma_h]_F \nu_F\|_{L^2(F)}^2 +
	\sum_{F \in \Fcal_\mathrm{N}} h_F\|g - \sigma_h \nu_F\|_{L^2(F)}^2.
\end{align}
All the constants $C_\mathrm{qb}$, $C_\mathrm{rel}$, $C_\mathrm{eff}$, and those throughout this paper (unless explicitly stated otherwise) do not depend on $\lambda$, $\mu$, and the mesh size.
In the computations, we chose $v = \mathcal{A} u_h$ in \eqref{def:eta} with the nodal average $\mathcal{A} u_h$ of the potential reconstruction $\mathcal{R} u_h$ of $u_h$. This choice is efficient as a non-trivial generalization to \cite{BertrandCarstensenGraessle2021} for the Poisson model problem.

\subsection{Outline}
The remaining parts of this paper are organized as follows.
\Cref{sec:discrete-problem} introduces the discrete problem with the aforementioned reconstruction operators and stabilization.
Quasi-best approximation and $L^2$ error estimates follow in \Cref{sec:a-priori}.
\Cref{sec:a-posteriori} derives 
reliable and efficient a~posteriori error estimates. Computational benchmarks
in \Cref{sec:numerical-examples} provide numerical evidence for $\lambda$-robustness.

\subsection{General notation}
Standard notation for Sobolev and Lebesgue spaces applies throughout this paper. In particular, $(\bullet, \bullet)_{L^2(\Omega)}$ denotes the $L^2$ scalar product and induces the norm $\|\bullet\| \coloneqq \|\bullet\|_{L^2(\Omega)}$ in $L^2(\Omega)$ and any product $L^2(\Omega)^n$ or $L^2(\Omega)^{n \times n}$ thereof.
The norm $\|\bullet\|_{-1}$ in the dual space $H^{-1}(\Omega)$ of $H^1_0(\Omega)$ reads
\begin{align}\label{def:dual-norm}
	\|F\|_{-1} \coloneqq \sup\nolimits_{v \in H^1_0(\Omega)\setminus\{0\}} F(v)/\|\nabla v\|.
\end{align}
For any $\Phi,\Psi \in L^2(\Omega)^{n \times n}$ and $s \in \R$, abbreviate $(\Phi, \Psi)_{\mathbb{C}^s} \coloneqq (\mathbb{C}^s \Phi, \Psi)_{L^2(\Omega)}$ and $\|\Psi\|_{\mathbb{C}^s}^2 \coloneqq (\Psi, \Psi)_{\mathbb{C}^s}$.
The symmetric and asymmetric part of the gradient $\D v$ of some function $v$ are denoted by $\varepsilon(v) \coloneqq \mathrm{sym}(\D v)$ and $\D_\mathrm{ss} v \coloneqq \mathrm{asym}(\D v)$ with the formula $\mathrm{sym}(M) \coloneqq (M + M^t)/2$ and $\mathrm{asym}(M) \coloneqq (M - M^t)/2$ for any $M \in \mathbb{R}^{n \times n}$. For Sobolev functions $v \in H^1(\Omega)^n$,
the Korn inequality
\begin{align}\label{ineq:Korn-inequality}
	\|v\|_{H^1(\Omega)} \leq C_\mathrm{K} \|\varepsilon (v)\|
\end{align}
holds  subject to different constraints to fix the rigid body motions. For $v \in V$, \eqref{ineq:Korn-inequality} is referred as the first Korn inequality with a constant $C_\mathrm{K}$ that exclusively depends on $\Omega$ and $\Gamma_{\mathrm{D}}$.
The second Korn inequality is \eqref{ineq:Korn-inequality} for any $v \in H^1(\Omega)^n$ with
\begin{align}\label{ineq:second-Korn}
	\int_\Omega v \d{x} = 0 \quad\text{and}\quad \int_\Omega \D_\mathrm{ss} v \d{x} = 0.
\end{align}
This is applied to piecewise $H^1$ vector fields, where \eqref{ineq:second-Korn} holds piecewise. Then $C_\mathrm{K}$ exclusively depends on $n$ and the shape regularity of the triangulation.
A proof of all this is provided in \cite{Brenner2004}.
The set of symmetric matrices is denoted by $\mathbb{S} \equiv \mathrm{sym} (\R^{n \times n})$.
Given two vectors $a, b \in \mathbb{R}^n$, recall the vector product $a \otimes b \coloneqq a b^t \in \R^{n \times n}$.
The deviatoric part of $M \in \mathbb{R}^{n \times n}$ is defined by $\mathrm{dev} M \coloneqq M - \tr(M) \mathrm{I}_{n \times n}/n$ with the trace $\mathrm{tr}(M) \coloneqq M : \mathrm{I}_{n \times n}$.
The notation $A \lesssim B$ means $A \leq C B$ for a generic constant $C$, $A \approx B$ abbreviates $A \lesssim B$ and $B \lesssim A$.

\section{Discrete problem}\label{sec:discrete-problem}
This section recalls the HHO methodology introduced in \cite{DiPietroErnLemaire2014,DiPietroErn2015,DiPietroDroniou2020,CicuttinErnPignet2021} for the linear elastic model problem of \Cref{sec:model}.

\subsection{Triangulation}\label{sec:triangulation}
Let $\Tcal$ be a regular triangulation into simplices with the set of sides $\Fcal$.
The set of interior sides (resp.~boundary sides) is denoted by $\mathcal{F}(\Omega)$ (resp.~$\Fcal(\partial \Omega) \coloneqq \Fcal\setminus\Fcal(\Omega)$).
We assume that the Dirichlet boundary $\Gamma_\mathrm{D}$ can be exactly resolved by the triangulation, i.e., the set $\Fcal_\mathrm{D} \coloneqq \{F \in \Fcal: F \subset \Gamma_\mathrm{D}\}$ of Dirichlet sides covers $\Gamma_\mathrm{D} = \cup \Fcal_\mathrm{D}$.
Fix the orientation of
the normal vector $\nu_F$ of an interior side $F \in \Fcal(\Omega)$ and $\nu_F \coloneqq \nu|_F$ for any boundary side $F \in \Fcal(\partial \Omega)$.
Given $F \in \Fcal(\Omega)$,
$T_\pm \in \Tcal$ denotes the unique simplex with $F \subset \partial T_\pm$ and $\nu_{T_\pm}|_F = \pm\nu_F$.
The jump $[v]_F$ of any function $v : \Omega \to \mathbb{R}$ with $v|_{T_\pm} \in W^{1,1}(\mathrm{int}(T_\pm))$ along $F \in \Fcal(\Omega)$ is defined by $[v]_F \coloneqq v|_{T_+} - v|_{T_-} \in L^1(F;\R^m)$.
If $F \in \Fcal(\partial \Omega)$, then $[v]_F \coloneqq v|_F$.
For any $T \in \Tcal$ (resp.~$F \in \Fcal$), define the simplex patch $\Omega(T) \coloneqq \mathrm{int}(\cup_{K \in \mathcal{T}, K \cap T \neq \emptyset} K)$ (resp.~side patch $\omega(F) \coloneqq \mathrm{int}(\cup_{T \in \mathcal{T}, F \in \Fcal(T)} T)$).
The set $H^1(\Tcal)$ is the space of all piecewise $H^1$ function with respect to the triangulation $\Tcal$.
The notation $\nabla_\pw$ denotes the piecewise application of the differential operator $\nabla$ (even without explicit reference to $\Tcal$).
This notation applies to the differential operators $\div$ and $\varepsilon$ as well.
In context of mesh-refining algorithms, there is an initial triangulation $\mathcal{T}_0$ with the set $\mathcal{V}_0$ (resp.~$\Fcal_0$) of vertices (resp.~sides) of $\Tcal_0$.
We assume that $\mathcal{T}_0$ satisfies the initial condition (IC) from \cite[Section 4, (a)--(b)]{Stevenson2008} to guarantee that the triangulations generated by successive refinements of $\Tcal_0$ with the newest-vertex-bisection (NVB) algorithm \cite{Maubach1995,Stevenson2008} are shape-regular.
The set of all such triangulations is denoted by $\mathbb{T}$ and $\Tcal \in \mathbb{T}$ is understood throughout the remaining parts of this paper.

\subsection{Finite element spaces}\label{sec:fem-spaces}
Given a subset $M \subset \R^n$ of diameter $h_M$, let $P_k(M)$ denote the space of polynomials on $M$ of total degree at most $k$.
For any $v \in L^1(M)$, $\Pi_M^k v \in P_k(M)$ denotes the $L^2$ projection of $v$ onto $P_k(M)$. The space of piecewise polynomials of degree at most $k$ with respect to $\mathcal{T}$ (resp.~$\Fcal$) is denoted by $P_k(\mathcal{T})$ (resp.~$P_k(\Fcal))$.
The continuous version of $P_{k+1}(\Tcal)^n$ with vanishing boundary data on $\Gamma_{\mathrm{D}}$
reads $S_\mathrm{D}^{k+1}(\Tcal) \coloneqq P_{k+1}(\Tcal)^n \cap V$.
Given $v \in L^1(\Omega)$ (resp.~$v \in L^1(\cup \Fcal)$), the $L^2$ projection $\Pi_{\Tcal}^k v$ (resp.~$\Pi_{\Fcal}^k v$) of $v$ onto $P_k(\Tcal)$ (resp.~$P_k(\Fcal)$) is $(\Pi_{\Tcal}^k v)|_T = \Pi_T^k v|_T$ in $T \in \Tcal$ (resp.~$(\Pi_{\Fcal}^k v)|_F = \Pi_F^k v|_F$ along $F \in \Fcal$).
The piecewise constant mesh-size function $h_\Tcal \in P_0(\Tcal)$ reads $h_\Tcal|_T = h_T$; $h_{\max} \coloneqq \max_{T \in \Tcal} h_T$ is the maximal simplex diameter in $\Tcal$.
For any $f \in L^2(\Omega)^n$ and $g \in L^2(\Gamma_\mathrm{N})^n$, define $\mathrm{osc}(f,\Tcal) \coloneqq \|h_\Tcal(1 - \Pi_{\Tcal}^k) f\|$ and $\mathrm{osc}(g, \Fcal_\mathrm{N})^2 \coloneqq \sum_{F \in \Fcal_\mathrm{N}} h_F\|(1 - \Pi_F^k) g\|_{L^2(F)}^2$.

\subsection{Discrete spaces}\label{sec:discrete-spaces}
Given a fixed natural number $k \geq 1$, let
\begin{align}\label{def:discrete-space}
	V_h \coloneqq P_k(\Tcal)^n \times P_k(\Fcal\setminus\Fcal_\mathrm{D})^n
\end{align}
denote the discrete ansatz space for $V$.
In this definition, $P_k(\Fcal\setminus\Fcal_\mathrm{D})^n$ is a subspace of $P_k(\Fcal)^n$ by zero extension: $v_\Fcal|_F \equiv 0$ on Dirichlet sides $F \in \Fcal_\mathrm{D}$ to model the homogeneous Dirichlet boundary condition of $v_\Fcal \in P_k(\Fcal\setminus\Fcal_\mathrm{D})^n$.
We note that the HHO method in this paper is \emph{not} well-posed for $k = 0$ and refer to \cite{BottiDiPietroGuglielmana2019} for the lowest-order case. 
For any $v_h = (v_\Tcal, v_\Fcal) \in V_h$, we abbreviate $v_T \coloneqq v_\Tcal|_T$ in a simplex $T \in \Tcal$ and $v_F \coloneqq v_\Fcal|_F$ along a side $F \in \Fcal$.
A norm in $V_h$ reads
\begin{align}\label{def:norm-h}
	\|v_h\|_h^2 \coloneqq \|\varepsilon_\pw(v_\mathcal{T})\|^2_{L^2(\Omega)} + \sum_{T \in \Tcal}\sum_{F \in \Fcal(T)} h_F^{-1}\|v_F - v_T\|_{L^2(F)}^2.
\end{align}
The interpolation operator $\I: V \to V_h$ maps $v \in V$ to $\I v \coloneqq (\Pi_\Tcal^k v, \Pi_\Fcal^k v) \in V_h$.

\subsection{Reconstructions and stabilization}
The potential reconstruction $\mathcal{R} : V_h \to P_{k+1}(\Tcal)^n$ maps $v_h = (v_\Tcal, v_\Fcal) \in V_h$ onto $\mathcal{R} v_h \in P_{k+1}(\Tcal)^n$ such that 
\begin{align}\label{def:potential-reconstruction-1}
	\begin{split}
		&\int_\Omega \varepsilon_\pw(\mathcal{R} v_h) : \varepsilon_\pw (\varphi_{k+1}) \d{x} \\
		&\qquad = -\int_\Omega v_\Tcal \cdot \div_\pw \varepsilon_\pw (\varphi_{k+1}) \d{x} + \sum_{F \in \Fcal} \int_F v_F \cdot [\varepsilon_\pw (\varphi_{k+1})]_F \nu_F \d{s}
	\end{split}
\end{align}
for any $\varphi_{k+1} \in P_{k+1}(\Tcal)^n$.
The system \eqref{def:potential-reconstruction-1} of linear equations defines $\mathcal{R} v_h$ uniquely up to the components associated with rigid-body motions. The latter are fixed by
\begin{align}\label{def:potential-reconstruction-2}
	\int_T \mathcal{R} v_h \d{x} = \int_T v_T \d{x} \text{ and } \int_T \D_\mathrm{ss} \mathcal{R} v_h \d{x} = \sum_{F \in \Fcal(T)} \int_F \mathrm{asym}(\nu_T \otimes v_F) \d{s}
\end{align}
for any $T \in \Tcal$.
The gradient reconstruction $\mathcal{G} v_h \in P_k(\Tcal)^{n \times n}$ of $v_h$ solves
\begin{align}\label{def:sym-grad-rec}
	\int_\Omega \mathcal{G} v_h : \Phi_k \d{x} = -\int_\Omega v_\Tcal \cdot \div_\pw \Phi_{k} \d{x} + \sum_{F \in \Fcal} \int_F v_F \cdot [\Phi_{k}]_F \nu_F \d{s}
\end{align}
for any $\Phi_k \in P_k(\Tcal)^{n \times n}$.
The symmetric part of $\mathcal{G} v_h$ is denoted by $\varepsilon_h v_h \coloneqq \mathrm{sym}(\mathcal{G} v_h) \in P_k(\Tcal;\mathbb{S})$.
It is shown in \cite[Section 7.2.5]{DiPietroDroniou2020} that $\varepsilon_h$ coincides with the symmetric gradient reconstruction of $v_h$ from \cite{BottiDiPietroSochala2017}, while the trace of $\varepsilon_h v_h$ is the divergence reconstruction of $v_h$ from \cite{DiPietroErn2015}.
With the difference operators
\begin{align}\label{def:difference-operator}
	\delta_{T}^k v_h \coloneqq \Pi_T^k(v_T - \mathcal{R} v_h) \in P_k(T)^n \text{ and } \delta_{TF}^k v_h \coloneqq \Pi_F^k(v_F - \mathcal{R} v_h|_T) \in P_k(F)^n
\end{align}
for any simplex $T \in \Tcal$ and side $F \in \Fcal(T)$ \cite[Section 2.1.4]{DiPietroDroniou2020},
the local stabilization $\s_T(u_h,v_h)$ of $u_h = (u_\Tcal, u_\Fcal), v_h = (v_\Tcal, v_\Fcal) \in V_h$ from \cite{DiPietroErn2015,BottiDiPietroSochala2017} reads
\begin{align}\label{def:stabilization}
	\s_T(u_h, v_h) \coloneqq \sum_{F \in \Fcal(T)} h_F^{-1} \int_F (\delta_{TF}^k u_h - \delta_{T}^k u_h) \cdot (\delta_{TF}^k v_h - \delta_{T}^k v_h) \d{s}.
\end{align}
The global weighted and non-weighted version of \eqref{def:stabilization}, namely
$$
\s(u_h, v_h) \coloneqq \sum_{T \in \Tcal} \mu|_T\s_T(u_h,v_h)\quad\text{and}\quad \widehat{\s}(u_h, v_h) \coloneqq \sum_{T \in \Tcal} \s_T(u_h,v_h)
$$
for any $u_h,v_h \in V_h$,
induce the seminorms $|\bullet|_\s \coloneqq \s(\bullet,\bullet)^{1/2}$ and $|\bullet|_{\widehat{\s}} \coloneqq \widehat{\s}(\bullet,\bullet)^{1/2}$ on $V_h$.
While the weighted version $\s$ with the Lam\'e parameter $\mu|_T > 0$ (from the isotropic elasticity tensor $\mathbb{C}$) is utilized in the numerical scheme, $\widehat{\s}$ serves only for theoretical purposes.
The reconstructions $\mathcal{R}$, $\mathcal{G}$, and the stabilization $\s$ can be computed simplex-wise and, therefore, in parallel.

\subsection{New equivalence of stabilizations}
This subsection proves the local equivalence of $\s$ and
the alternative stabilization $\widetilde{\s}$ defined by $\widetilde{\s}(u_h,v_h) = \sum_{T \in \Tcal} \widetilde{\s}_T(u_h,v_h)$ and
\begin{align}
	\label{def:alternative-stabilization}
	\widetilde{\s}_T(u_h, v_h) &\coloneqq h_T^{-2} (\delta_{T}^k u_h, \delta_{T}^k v_h)_{L^2(T)} + \sum_{F \in \Fcal(T)} h_F^{-1}(\delta_{TF}^k u_h, \delta_{TF}^k v_h)_{L^2(F)}
\end{align}
for any $T \in \Tcal$ and $u_h, v_h \in V_h$.
This stabilization $\widetilde{\s}$ was utilized in an HHO method for the Poisson equation in \cite{DiPietroDroniou2020}, but its equivalence to the classical HHO stabilization from \cite{DiPietroErnLemaire2014} was established later \cite[Theorem 4]{BertrandCarstensenGraessle2021}. \Cref{lem:equiv-stab} below extends the equivalence $\s \approx \widetilde{\s}$ from \cite{BertrandCarstensenGraessle2021} to the linear elasticity model problem.
We note that \Cref{lem:equiv-stab} holds for general polytopal meshes.

\begin{theorem}[equivalence of stabilizations]\label{lem:equiv-stab}
	Any $v_h = (v_\Tcal, v_\Fcal) \in V_h$ and $T \in \Tcal$ satisfy $\s_T(v_h,v_h) \approx \widetilde{\s}_T(v_h,v_h)$.
\end{theorem}
\begin{proof}
	The assertion $\lesssim$ follows immediately from the triangle and the discrete trace inequality. The remaining parts of this proof are therefore devoted to the reverse direction $\gtrsim$.
	A triangle and a discrete trace inequality imply
	\begin{align}\label{ineq:proof-equiv-stab-split}
		\|\delta_{TF}^k v_h\|_{L^2(F)} \lesssim \|\delta_{TF}^k v_h - \delta_{T}^k v_h\|_{L^2(F)} + h_T^{-1/2}\|\delta_{T}^k v_h\|_{L^2(T)}
	\end{align}
	for any $F \in \Fcal(T)$. Therefore, it remains to prove $h_T^{-1}\|\delta_{T}^k v_h\|_{L^2(T)} \lesssim \s_T(u_h,v_h)^{1/2}$.
	Let $x_T$ denote the midpoint of $T$.
	Define the rigid-body motion $\varphi_{\mathrm{RM}} \coloneqq \Pi_T^0 \delta_T^k v_h + (\Pi_T^0 \D_\mathrm{ss} \delta_T^k v_h) (x - x_T)$ as in \cite{DiPietroErn2015}.
	Since $\Pi_T^0 \delta_T^k v_h = 0$ from \eqref{def:potential-reconstruction-2}, we infer
	\begin{align}\label{ineq:proof-equiv-stab-L2-RM}
		h_T^{-1}\|\varphi_\mathrm{RM}\|_{L^2(T)} \approx \|\Pi_T^0 \D_\mathrm{ss} \delta_T^k v_h\|_{L^2(T)} = |T|^{-1/2}\Big|\int_T \D_\mathrm{ss} \delta_T^k v_h \d{x}\Big|.
	\end{align}
	The outer normal vector $\nu_T$ is constant along any $F \in \Fcal(T)$ and hence,
	the convention \eqref{def:potential-reconstruction-2} in the definition of $\mathcal{R}$ and an integration by parts show
	\begin{align*}
		\sum_{F \in \Fcal(T)} \int_F \mathrm{asym}(\nu_T \otimes \delta_{TF}^k v_h) \d{s} = \sum_{F \in \Fcal(T)} \int_F \mathrm{asym}(\nu_T \otimes (v_F - \mathcal{R} v_h|_T)) \d{s} = 0.
	\end{align*}
	In combination with another integration by parts, we obtain
	\begin{align*}
		\int_T \D_\mathrm{ss} \delta_{T}^k v_h \d{x} &= \sum_{F \in \Fcal(T)} \int_F \mathrm{asym}(\nu_T \otimes \delta_{T}^k v_h) \d{s}\\
		& = \sum_{F \in \Fcal(T)} \int_F \mathrm{asym}(\nu_T \otimes (\delta_{T}^k v_h - \delta_{TF}^k v_h)) \d{s}.
	\end{align*}
	This, \eqref{ineq:proof-equiv-stab-L2-RM}, and a Cauchy inequality reveal
	\begin{align*}
		h_T^{-1}\|\varphi_\mathrm{RM}\|_{L^2(T)} \lesssim \sum_{F \in \Fcal(T)} h_F^{-1/2} \|\delta_{TF}^k v_h - \delta_{T}^k v_h\|_{L^2(F)} \lesssim \s_T(v_h,v_h)^{1/2}.
	\end{align*}
	Since $h_T^{-1}\|\delta_{T}^k v_h - \varphi_{\mathrm{RM}}\|_{L^2(T)} \lesssim \|\varepsilon (\delta_{T}^k v_h)\|_{L^2(T)}$ from a Poincar\'e and the (second) Korn inequality \eqref{ineq:Korn-inequality},
	\begin{align}\label{ineq:proof-equiv-stab-split-RM}
		h_T^{-1}\|\delta_{T}^k v_h\|_{L^2(T)} \lesssim \s_T(v_h,v_h)^{1/2} + \|\varepsilon (\delta_{T}^k v_h)\|_{L^2(T)}.
	\end{align}
	The proof of $\|\varepsilon (\delta_{T}^k v_h)\|_{L^2(T)}^2 \lesssim s_T(v_h,v_h)$ below generalizes that of \cite[Theorem 4]{BertrandCarstensenGraessle2021}.
	An integration by parts provides
	\begin{align}\label{ineq:proof-equiv-stab-L2-vol-1}
		\|\varepsilon (\delta_{T}^k v_h)\|^2_{L^2(T)} &= -(\delta_{T}^k v_h, \div\, \varepsilon (\delta_{T}^k v_h))_{L^2(T)} + (\delta_{T}^k v_h, \varepsilon (\delta_{T}^k v_h) \nu_T)_{L^2(\partial T)}.
	\end{align}
	Since $\div\, \varepsilon (\delta_{T}^k v_h) \in P_{k-2}(T)^n \subset P_k(T)^n$ with the convention $P_{-1}(T) = \{0\}$, another integration by parts and the definition \eqref{def:potential-reconstruction-1} of $\mathcal{R}$ result in
	\begin{align*}
		\begin{split}
			&-(\Pi_T^k \mathcal{R} v_h, \div\, \varepsilon (\delta_{T}^k v_h))_{L^2(T)} = -(\mathcal{R} v_h, \div\, \varepsilon (\delta_{T}^k v_h))_{L^2(T)}\\
			&\qquad = -(v_T, \div\, \varepsilon (\delta_{T}^k v_h))_{L^2(T)} + \sum_{F \in \Fcal(T)} (v_F - \mathcal{R} v_h|_T, \varepsilon(\delta_{T}^k v_h) \nu_T)_{L^2(F)}.
		\end{split}
	\end{align*}
	Recall $\delta_{T}^k v_h = \Pi_T^k(v_T - \mathcal{R} v_h)$ to rewrite this as
	\begin{align}\label{ineq:proof-equiv-stab-L2-vol-2}
		(\delta_{T}^k v_h, \div\, \varepsilon (\delta_{T}^k v_h))_{L^2(T)} = \sum_{F \in \Fcal(T)} (\delta_{TF}^k v_h, \varepsilon(\delta_{T}^k v_h) \nu_T)_{L^2(F)}
	\end{align}
	with $\varepsilon(\delta_{T}^k v_h) \nu_T \in P_k(F)$ along any $F \in \Fcal(T)$ in the last step.
	The combination of \eqref{ineq:proof-equiv-stab-L2-vol-1}--\eqref{ineq:proof-equiv-stab-L2-vol-2} with a Cauchy and discrete trace inequality reveals
	\begin{align}
		\|\varepsilon (\delta_{T}^k v_h)\|^2_{L^2(T)} \lesssim \sum_{F \in \Fcal(T)} h_F^{-1/2}\|\delta_{TF}^k v_h - \delta_{T}^k v_h\|_{L^2(F)} \|\varepsilon(\delta_{T}^k v_h)\|_{L^2(T)}.
	\end{align}
	Hence, $\|\varepsilon (\delta_{T}^k v_h)\|_{L^2(T)}^{2} \lesssim s_T(v_h,v_h)$, \eqref{ineq:proof-equiv-stab-split}, and \eqref{ineq:proof-equiv-stab-split-RM} conclude the proof.
\end{proof}

While \Cref{lem:equiv-stab} also holds for general polytopal meshes, it provides a clear description of the kernel of $|\bullet|_\s$ on triangulations into simplices.
Define the set $\mathrm{CR}^{k+1}_\mathrm{D}(\Tcal) \coloneqq \{v_{k+1} \in P_{k+1}(\Tcal)^n : \Pi_F^k [v_{k+1}]_F = 0 \text{ for any } F \in \Fcal \setminus \Fcal_\mathrm{N}\}$ of Crouzeix-Raviart functions.
The interpolation $\I : V \to V_h$ can be extended to $\I : V + \mathrm{CR}_\mathrm{D}^{k+1}(\Tcal) \to V_h$ because $\Pi_F^k v_\mathrm{CR}$ is uniquely defined for Crouzeix-Raviart functions.

\begin{corollary}[kernel of $|\bullet|_\s$]\label{cor:kernel-stabilization}
	Any given $v_h \in V_h$ satisfies $|v_h|_\s = 0$ if and only if $v_h = \I v_\mathrm{CR}$ holds for some $v_\mathrm{CR} \in \mathrm{CR}^{k+1}_\mathrm{D}(\Tcal)$.
\end{corollary}

\begin{proof}
	If $v_h \in V_h$ with $|v_h|_\s = 0$, then \Cref{lem:equiv-stab} implies that $\Pi_{\Tcal}^k \mathcal{R} v_h = v_\Tcal$ and $\Pi_F^k \mathcal{R} v_h|_T = v_F$ for any $T \in \Tcal$ and $F \in \Fcal(T)$.
	This and the convention $v_F \equiv 0$ on $F \in \Fcal_\mathrm{D}$ shows $\Pi_F^k [\mathcal{R} v_h]_F = 0$ for any $F \in \Fcal\setminus\Fcal_\mathrm{N}$ and so, $\mathcal{R} v_h \in \mathrm{CR}_\mathrm{D}^{k+1}(\Tcal)$ with $\I \mathcal{R} v_h = v_h$.
	On the other hand, let $v_\mathrm{CR} \in \mathrm{CR}^{k+1}_\mathrm{D}(\Tcal)$ be given.
	A piecewise integration by parts proves that $v_\mathrm{CR} \in P_{k+1}(\Tcal)^n$ satisfies \eqref{def:potential-reconstruction-1}--\eqref{def:potential-reconstruction-2} with $v_\Tcal \coloneqq \Pi_\Tcal^k v_\mathrm{CR}$ and $v_F \coloneqq \Pi_F^k v_\mathrm{CR}$ for any $F \in \Fcal$, whence
	\begin{align}\label{eq:RI=Id}
		\mathcal{R} \I v_\mathrm{CR} = v_\mathrm{CR}.
	\end{align}
	This and \Cref{lem:equiv-stab} conclude the proof of $|\mathrm{I} v_\mathrm{CR}|_\s = 0$.
\end{proof}

\subsection{Elementary properties}
The following results state the characteristic properties of reconstruction operators in the HHO methodology.
\begin{lemma}[commuting diagram]\label{lem:commuting-diagram}
	Any $v \in V$ and $s \in \mathbb{R}$ satisfy
	\begin{align*}
		\Pi_\Tcal^k \varepsilon (v) = \varepsilon_h \I v \quad\text{and}\quad \Pi_\Tcal^k \mathbb{C}^s \varepsilon (v) = \mathbb{C}^s \varepsilon_h \I v.
	\end{align*}
\end{lemma}
\begin{proof} 
	The first identity follows from $\Pi_\Tcal^k \D v = \mathcal{G} \I v$ \cite[Eq.~(4.40)]{DiPietroDroniou2020} and is stated in \cite[Eq.~(18)]{BottiDiPietroSochala2017}. The second identity is an immediate consequence of the first for piecewise constant parameters $\lambda, \mu$.
\end{proof}
\begin{lemma}[best approximation of $\mathcal{R} \circ \I$]\label{lem:quasi-opt-R-I}
	Any $v \in V$ satisfies
	\begin{align}\label{ineq:ba-RI}
		\|\D_\pw(v - \mathcal{R} \I v)\| \lesssim \|\varepsilon_\pw(v - \mathcal{R} \I v)\| = \min_{v_{k+1} \in P_{k+1}(\Tcal)^n} \|\varepsilon_\pw(v - v_{k+1})\|.
	\end{align}
\end{lemma}
\begin{proof}
	From \eqref{def:potential-reconstruction-2}, we infer for any $v \in V$ and $T \in \Tcal$ that
	\begin{align*}
		\int_T (v - \mathcal{R} \I v) \d{x} = 0 \quad\text{and}\quad \int_T \D_\mathrm{ss}(v - \mathcal{R} \I v) \d{x} = 0.
	\end{align*}
	Hence, the (second) Korn inequality \eqref{ineq:Korn-inequality} proves the inequality in \eqref{ineq:ba-RI}. The equality in \eqref{ineq:ba-RI} follows from the $L^2$ orthogonality $\varepsilon_\pw(v - \mathcal{R} \I v) \perp \varepsilon_\pw (P_{k+1}(\Tcal)^n)$ \cite[Eq.~(19)]{DiPietroErn2015}.
\end{proof}
A straightforward consequence of \Cref{lem:commuting-diagram}--\ref{lem:quasi-opt-R-I} is the following.
\begin{remark}[$\mathcal{\varepsilon}_h \circ \I = \varepsilon$ in $S^{k+1}_\mathrm{D}(\Tcal)$]\label{rem:RI=Id}
	If $v \in S^{k+1}_\mathrm{D}(\Tcal)$, then $\varepsilon (v) \in \varepsilon(P_{k+1}(\Tcal)^n) \subset P_k(\Tcal)^{n \times n}$.
	\Cref{lem:commuting-diagram}--\ref{lem:quasi-opt-R-I} imply $\varepsilon(v) = \varepsilon_h \I v = \varepsilon_\pw(\mathcal{R} \I v)$.
\end{remark}
Another implication of \Cref{lem:equiv-stab} is the following bound.

\begin{lemma}[$\varepsilon_\pw \circ \mathcal{R}$ vs $\varepsilon_h$]\label{lem:epsR-epsh}
	Any $v_h \in V_h$ and $T \in \Tcal$ satisfy
	\begin{align*}
		\|\varepsilon(\mathcal{R} v_h) - \varepsilon_h v_h\|_{L^2(T)}^2 \lesssim \s_T(v_h,v_h).
	\end{align*}
\end{lemma}
\begin{proof}
	Given $v_h = (v_\Tcal, v_\Fcal) \in V_h$, abbreviate $\Phi_k \coloneqq (\varepsilon_\pw(\mathcal{R} v_h) - \varepsilon_h v_h)|_T \in P_k(T;\mathbb{S})$. The definition of $\varepsilon_h$ in \eqref{def:sym-grad-rec} and an integration by parts imply
	\begin{align*}
		\|\Phi_k\|^2_{L^2(T)} &= (\varepsilon(\mathcal{R} v_h) - \varepsilon_h v_h, \Phi_k)_{L^2(T)}\\
		& = (v_T - \mathcal{R} v_h, \div\, \Phi_k)_{L^2(T)} + \sum_{F \in \Fcal(T)} (\mathcal{R} v_h|_T - v_F, \Phi_k \nu_F)_{L^2(F)}.
	\end{align*}
	Since $\div\, \Phi_k \in P_k(T)^n$ and $\Phi_k \nu_F \in P_k(F)^n$, this, a Cauchy inequality, the inverse inequality $\|\div\, \Phi_k\|_{L^2(T)} \lesssim h_T^{-1}\|\Phi_k\|_{L^2(T)}$, and the discrete trace inequality $\|\Phi_k\|_{L^2(F)} \lesssim h_F^{-1/2}\|\Phi_k\|_{L^2(T)}$ for any $F \in \Fcal(T)$ provide
	\begin{align*}
		\|\Phi_k\|_{L^2(T)}^2 \lesssim \|&h_T^{-1}\delta_{T}^k v_h\|_{L^2(T)}^2 + \sum_{F \in \Fcal(T)} h_F^{-1}\|\delta_{TF}^k v_h\|_{L^2(F)}^2 = \widetilde{\s}_T(v_h, v_h).
	\end{align*}
	This and \Cref{lem:equiv-stab} conclude the proof.
\end{proof}

\subsection{Discrete formulation}
The discrete problem seeks the solution $u_h \in V_h$ to
\begin{align}\label{def:discrete-problem}
	a_h(u_h,v_h) = \int_\Omega f \cdot v_\Tcal \d{x} + \int_{\Gamma_\mathrm{N}} g \cdot v_\Fcal \d{s} \quad\text{for all } v_h = (v_\Tcal, v_\Fcal) \in V_h
\end{align}
with the discrete bilinear form
\begin{align}
	a_h(u_h,v_h) &\coloneqq (\mathbb{C} \varepsilon_h u_h, \varepsilon_h v_h)_{L^2(\Omega)} + \s(u_h,v_h)\label{def:a-h}.
\end{align}
The following result proves that $a_h$ is a scalar product in $V_h$ and the induced norm is denoted by $\|\bullet\|_{a_h}$.
Recall the norm $\|\bullet\|_h$ from \eqref{def:norm-h}.

\begin{theorem}[Existence and uniqueness of discrete solutions]\label{thm:norm-equivalence}
	There exist positive constants $\newcnst\label{cnst:norm-equivalence-1}, \newcnst\label{cnst:norm-equivalence-2}$ that solely depend on $k$ and the shape regularity of $\Tcal$ with
	\begin{align}\label{ineq:norm-equivalence}
		\cnst{cnst:norm-equivalence-1}^{-1}\|v_h\|_h \leq \|\varepsilon_h v_h\| + |v_h|_{\widehat{\s}} \approx \|\varepsilon_\pw (\mathcal{R} v_h)\| + |v_h|_{\widehat{\s}} \leq \cnst{cnst:norm-equivalence-2} \|v_h\|_h
	\end{align}
	for any $v_h \in V_h$.
	In particular, $a_h$ from \eqref{def:a-h} is a scalar product on $V_h$ and there exists a unique discrete solution to \eqref{def:discrete-problem}.
\end{theorem}
\begin{proof}
	This result is established in \cite[Section 7.2.6]{DiPietroDroniou2020}.
\end{proof}
Throughout the remaining parts of this paper, $u_h$ denotes the unique solution to the discrete problem \eqref{def:discrete-problem} and $\sigma_h \coloneqq \mathbb{C} \varepsilon_h u_h \in P_k(\Tcal;\mathbb{S})$ is
the discrete stress.
\subsection{Comments on alternative schemes}

\begin{remark}[equivalent stabilization]
	The results of this paper can be extended to \eqref{def:discrete-problem} with $\s$ replaced by $\widetilde{\s}$ (by equivalence of stabilizations in \Cref{lem:equiv-stab}).
\end{remark}

\begin{remark}[HDG stabilization]\label{rem:HDG-stab}
	Another HHO method utilizes the ansatz space $V_h \coloneqq P_{k+1}(\Tcal)^n \times P_k(\Fcal\setminus \Fcal_\mathrm{D})^n$ with the Lehrenfeld-Sch\"oberl stabilization
	\begin{align}\label{def:hdg-stab}
		\s_T^{\mathrm{HDG}}(u_h,v_h) \coloneqq \sum_{F \in \Fcal(T)} h_F^{-1} \int_F \Pi_F^k(u_F - u_T) \cdot \Pi_F^k (v_F - v_T) \d{s}
	\end{align}
	for $T \in \Tcal$
	from the HDG methodology \cite{Lehrenfeld2010}. 
	Straight-forward adaptations of the arguments in the proof of \Cref{lem:equiv-stab} provides the local equivalence
	\begin{align}\label{ineq:equiv-stab-hdg}
		\s_T^{\mathrm{HDG}}(v_h,v_h)^{1/2} \approx h_T^{-1}\|\delta_{T}^{(k+1)} v_h\|_{L^2(T)} + \sum_{F \in \Fcal(T)} h_F^{-1/2} \|\delta_{TF}^k v_h\|_{L^2(F)}.
	\end{align}
	This holds for any $v_h \in P_{k+1}(\Tcal)^n \times P_k(\Fcal\setminus \Fcal_\mathrm{D})^n$ and the modified difference operator
	$\delta_{T}^{(k+1)} v_h \coloneqq v_T - \mathcal{R} v_h \in P_{k+1}(T)^n$ and  $\delta_{TF}^k v_h$ from \eqref{def:difference-operator} with the representation $\Pi_F^k(v_F - v_T) = \Pi_F^k(\delta_{TF}^k v_h - \delta_T^{(k+1)} v_h)$ from
	\cite[Section 5.1.6]{DiPietroDroniou2020}.
\end{remark}
\begin{remark}[HDG error control]\label{rem:HDG}
	The equivalence \eqref{ineq:equiv-stab-hdg} allow for a generalization of the error analysis of this paper to the HHO scheme from \cite[Section 5.1.6]{DiPietroDroniou2020}, which can also be seen as an HDG method.
	In particular, the error estimates
	\eqref{ineq:a-priori}--\eqref{ineq:reliability} hold verbatim with the discrete ansatz space $V_h = P_{k+1}(\Tcal)^n \times P_k(\Fcal\setminus\Fcal_\mathrm{D})^n$ and the stabilization $\s^\mathrm{HDG}$ from \eqref{def:hdg-stab}.
\end{remark}
\begin{remark}[Comparison to conforming, nonconforming, and mixed FEM]
	An adaptive conforming $hp$ FEM is preferable if high accuracy is mandatory.
	For lower-order approximations, the analyzed scheme is one of the simplest $\lambda$-robust methods in 3D. (Notice that the nonconforming Kouhia-Stenberg FEM \cite{KouhiaStenberg1995} applies only in 2D). The Arnold-Winther mixed scheme \cite{ArnoldWinther2002} is more elaborate than the locking-free conforming FEM for $k \geq 4$ \cite{BabuskaSuri1992}. For weakly symmetric FEM, we refer to \cite{LedererStenberg2024} and the references quoted therein, but mention that strong symmetry is often desirable.
\end{remark}
\begin{remark}[Comparison within skeletal methods]
	The simplest VEM method for $k = 1$ is not looking-free \cite{virtualelements2023}. Compared to discontinuous Galerkin schemes, the HHO methodology is parameter-free. A close competitor is the HDG method with the Lehrenfeld-Sch\"oberl stabilization from \cite{Lehrenfeld2010}, cf.~\Cref{rem:HDG-stab}--\ref{rem:HDG}.
\end{remark}
\begin{remark}[Stokes equation]
	The incompressible limit as $\lambda \to \infty$ leads formally to the Stokes equations in the symmetric formulation (i.e.~in terms of the linear Green strain $\varepsilon$). The estimates in this paper are $\lambda$-robust and hence hold for the resulting
	Stokes problem as well and lead to a stable novel HHO discretisation with corresponding error terms \eqref{ineq:a-priori}--\eqref{ineq:reliability}.
	A similar extension is also
	expected for the known HHO scheme \cite{DiPietroErnLinkeSchienk2016} and their HDG relatives in the non-symmetric form
	from the textbooks \cite{GiraultRaviart1986,BrennerScott2008,BoffiBrezziFortin2013}
	(when $\Gamma_\mathrm{N} = \emptyset$, the linear Green strain $\varepsilon$ can be replaced by the gradient $\D$).
\end{remark}

\section{A priori error analysis}\label{sec:a-priori}
This section establishes the quasi-best approximation result \eqref{ineq:a-priori}.

\subsection{Main result}
Let $u \in V$ denote the exact solution to \eqref{def:continuous-problem} and $\sigma = \mathbb{C} \varepsilon(u)$.
To obtain $\lambda$-robust error estimates for the $L^2$ stress error $\|\sigma - \sigma_h\|$, the tr-dev-div lemma stated in \Cref{lem:tr-div-dev} below is utilized under the following assumption, which imposes
mild geometric assumptions on the initial triangulation $\Tcal_0$ for $|\Gamma_{\mathrm{N}}| > 0$. 
\begin{assumption}\label{assumption} 
	Either $|\Gamma_{\mathrm{N}}| = 0$ or there exists $\varphi_0 \in S^k_\mathrm{D}(\Tcal_0)$ with $\int_\Omega \div \varphi_0 \d{x} \neq 0$.
\end{assumption}

\begin{remark}[(\ref{assumption}) for $k < n$]
	Suppose that $z$ is a vertex of $\Tcal_0$ in the relative interior of $\Gamma_{\mathrm{N}}$, define $\varphi_0 \coloneqq a \varphi_z \in S^1_\mathrm{D}(\Tcal_0)$ with a vector $a \in \R^n$ and the nodal basis function $\varphi_z$ associated with $z$. Let $\Fcal_{z} \coloneqq \{F \in \Fcal_0: z \in F\} \subset \Fcal_{0,\mathrm{N}}$ denote the set of all sides containing $z$, where $\Fcal_{0,\mathrm{N}}$ is the set of Neumann sides of $\Tcal_0$. An integration by parts reveals
	\begin{align}\label{eq:assumption-sufficient-condition}
		n\int_\Omega \div\, \varphi_0 \d{x} = a \cdot \sum_{F \in \Fcal_{z}} |F| \nu_F.
	\end{align}
	In 2D, there are $|\Fcal_{z}| = 2$ edges and $a \coloneqq \sum_{F \in \Fcal_{z}} |F| \nu_F \neq 0$ leads in \eqref{eq:assumption-sufficient-condition} to (\ref{assumption}).
	In 3D, the same arguments lead to (\ref{eq:assumption-sufficient-condition}) for a vertex $z$ in a flat part of $\Gamma_{\mathrm{N}}$ or $z$ is a convex/concave corner point.
	Notice that this can always be generated by mesh-refinements.
\end{remark}

\begin{remark}[(\ref{assumption}) for $k \geq n$]
	Given $F \in \Fcal_{0,\mathrm{N}}$, consider the face bubble function $b_F \coloneqq \prod_{z \in \mathcal{V}_0(F)} \varphi_z$ with the set $\mathcal{V}_0(F)$ of all vertices of $F$. For $\varphi_0 \coloneqq b_F \nu_F \in S^n_\mathrm{D}(\Tcal_0)$, an integration by parts provides
	\begin{align*}
		\int_\Omega \div\, \varphi_0 \d{x} = \int_F b_F \d{x} > 0.
	\end{align*}
	In other words, $k \geq n$ implies (\ref{assumption}) without any additional geometric assumption.
\end{remark}
The following result implies the convergence rates from \cite{DiPietroErn2015,DiPietroDroniou2020}.

\begin{theorem}[quasi-best approximation]\label{thm:best-approximation}
	Suppose \eqref{assumption}, then the discrete solution $u_h$ to \eqref{def:discrete-problem} and $\sigma_h = \mathbb{C} \varepsilon_h u_h$ satisfy \eqref{ineq:a-priori}.
	The $\lambda$-independent constant $C_\mathrm{qb}$ exclusively depends on $k$, $\Gamma_{\mathrm{D}}$, $\Sigma_0$, and $\mathbb{T}$,
	where
	\begin{align}\label{def:Sigma-0}
		\Sigma_0 \coloneqq \begin{cases}
			\{\tau \in L^2(\Omega)^{n \times n} : \int_\Omega \tr\, \tau \d{x} = 0\} \text{ if } \Gamma_{\mathrm{N}} = \emptyset \text{ and else}\\
			\{\tau \in L^2(\Omega)^{n \times n} : (\tau, \varepsilon(\varphi_0))_{L^2(\Omega)} = 0\}
		\end{cases}
	\end{align}
	is a closed subspace
	of $L^2(\Omega)^{n \times n}$ with $\Tcal_0$ and $\varphi_0$ from (\ref{assumption}).
\end{theorem}

Several preliminary results precede the proof of \Cref{thm:best-approximation} in \Cref{sec:proof-a-priori}.

\subsection{Right inverse}
This subsection recalls the right inverse $\mathcal{J}$ of $\I$
from \cite{ErnZanotti2020}.
Let $\mathcal{A} : V_h \to S^{k+1}_\mathrm{D}(\Tcal)$ denote the averaging operator that maps $v_h \in V_h$ onto $\mathcal{A} v_h \in S^{k+1}_\mathrm{D}(\Tcal)$ by averaging all possible values of $\mathcal{R} v_h$ at the nodal degrees of freedom of $S^{k+1}_\mathrm{D}(\Tcal)$.
\begin{lemma}[right-inverse]\label{lem:conforming-companion}
	There exists a linear operator $\mathcal{J} : V_h \to V$ such that $\mathrm{I} \circ \mathcal{J} = \mathrm{Id}$ and any $v_h \in V_h$ satisfies (a) $\|\varepsilon_\pw(\mathcal{J}v_h - \mathcal{R}v_h)\| \lesssim \inf_{v \in V} \|\D_\pw(v - \mathcal{R} v_h)\| + |v_h|_{\widehat{\s}}$ and (b) $\mu_0^{1/2}\|\varepsilon (\mathcal{J} v_h)\| \lesssim \|v_h\|_{a_h}$.
\end{lemma}
\begin{proof}[Proof of \Cref{lem:conforming-companion}.a]
	The construction of $\mathcal{J}$ is provided in \cite{ErnZanotti2020} for the scalar case with homogeneous boundary condition along the entire body. The extension to the vector-valued case by component-wise application of the operator therein is straightforward with a minor modification in \cite[Def.~(4.19)]{ErnZanotti2020} owing to possible Neumann boundary conditions: The set of interior sides $\Fcal(\Omega)$ is replaced by $\Fcal\setminus\Fcal_\mathrm{N}$ to enforce homogeneous boundary data only on $\Gamma_{\mathrm{D}}$.
	The estimate
	\begin{align}\label{ineq:proof-conf-companion-1}
		\|\D_\pw(\mathcal{R} v_h - \mathcal{J} v_h)\|^2 \lesssim \sum_{F \in \Fcal\setminus\Fcal_\mathrm{N}} h_F^{-1}\|[\mathcal{R} v_h]_F\|^2_{L^2(F)} + \widetilde{\s}(v_h,v_h)
	\end{align}
	follows from the arguments in the proof of \cite[Proposition 4.7]{ErnZanotti2020}, where \cite[Proposition 4.6]{ErnZanotti2020} applies to $(\Pi_\Tcal^k(v_\Tcal - \mathcal{A} v_h), (\Pi_F^k(v_F - \mathcal{A} v_h))_{F \in \Fcal\setminus\Fcal_\mathrm{D}})$ instead of $(v_\Tcal - \mathcal{A} v_h, (v_F - \mathcal{A} v_h)_{F \in \Fcal\setminus\Fcal_\mathrm{D}})$ in the ultimate formula of \cite[p.~2179]{ErnZanotti2020}.
	A triangle and a Poincar\'e inequality along $F \in \mathcal{F}$ imply
	\begin{align}\label{ineq:proof-conf-companion-2}
		\|[\mathcal{R} v_h]_F\|_{L^2(F)} \lesssim h_F\|[\D_\pw \mathcal{R} v_h \times \nu_F]_F\|_{L^2(F)} + \|[\Pi_F^0 \mathcal{R} v_h]_F\|_{L^2(F)}.
	\end{align}
	The bubble-function techniques \cite{Verfuerth2013} lead to the efficiency estimate
	\begin{align}\label{ineq:proof-conf-companion-3}
		h^{1/2}_F\|[\D_\pw \mathcal{R} v_h \times \nu_F]_F\|_{L^2(F)} \lesssim \inf_{v \in V} \|\D_\pw(v - \mathcal{R} v_h)\|_{L^2(\omega(F))},
	\end{align}
	cf.~\cite[Lemma 7]{BertrandCarstensenGraessle2021} for further details.
	The stability of the $L^2$ projection and a triangle inequality reveal
	\begin{align*}
		\|[\Pi_F^0 \mathcal{R} v_h]_F\|_{L^2(F)} \leq \|[\Pi_F^k \mathcal{R} v_h]_F\|_{L^2(F)} \leq \sum_{\{T \in \Tcal : F \in \Fcal(T)\}} \|\Pi_F^k(v_F - \mathcal{R} v_h|_T)\|_{L^2(F)}.
	\end{align*}
	The combination of this with \eqref{ineq:proof-conf-companion-1}--\eqref{ineq:proof-conf-companion-3} results in
	\begin{align*}
		\|\D_\pw(\mathcal{R} v_h - \mathcal{J} v_h)\|^2 \lesssim \inf_{v \in V} \|\D_\pw(v - \mathcal{R} v_h)\|^2 + \widetilde{\s}(v_h,v_h).
	\end{align*}
	This, \Cref{lem:equiv-stab}, and $\|\varepsilon_\pw(\mathcal{R} v_h - \mathcal{J} v_h)\| \leq \|\D_\pw(\mathcal{R} v_h - \mathcal{J} v_h)\|$ prove (a).\vspace*{0.2cm}
	
	\noindent\emph{Proof of \Cref{lem:conforming-companion}.b.}
	Recall the norm $\|\bullet\|_h$ from \eqref{def:norm-h}.
	Given $F \in \Fcal$, the triangle inequality provides $\|[\mathcal{R} v_h]_F\|_{L^2(F)} \leq \|[\mathcal{R} v_h - v_\Tcal]_F\|_{L^2(F)} + \|[v_\Tcal]_F\|_{L^2(F)}$ and $\|[v_\Tcal]_F\|_{L^2(F)} \leq \sum_{\{T \in \Tcal, F \in \Fcal(T)\}} \|v_T - v_F\|_{L^2(T)}$. This, \eqref{ineq:proof-conf-companion-1}, and a discrete trace inequality lead to
	\begin{align*}
		\|\D_\pw(\mathcal{R} v_h - \mathcal{J} v_h)\| \lesssim \|h_\Tcal^{-1}(\mathcal{R} v_h - v_\Tcal)\|_{L^2(\Omega)} + \widetilde{\s}(v_h,v_h)^{1/2} + \|v_h\|_h.
	\end{align*}
	In combination with $\|h_\Tcal^{-1}(\mathcal{R} v_h - \Pi_\Tcal^k \mathcal{R} v_h)\|_{L^2(\Omega)} \lesssim \|\varepsilon_\pw(\mathcal{R} v_h)\|$ from \cite[p.~8]{DiPietroErn2015} and a triangle inequality, we infer
	\begin{align}\label{ineq:proof-conf-companion-b}
		\|\D_\pw(\mathcal{R} v_h - \mathcal{J} v_h)\| \lesssim \|\varepsilon_\pw(\mathcal{R} v_h)\| + \widetilde{\s}(v_h,v_h)^{1/2} + \|v_h\|_h.
	\end{align}
	\Cref{lem:equiv-stab} and \ref{thm:norm-equivalence} imply
	$\|\varepsilon_\pw(\mathcal{R} v_h)\| + \widetilde{\s}(v_h,v_h)^{1/2} + \|v_h\|_h \approx \|\varepsilon_h v_h\| + |v_h|_{\widehat{\s}}$. This, $\mu_0^{1/2}(\|\varepsilon_h v_h\| + |v_h|_{\widehat{\s}})\lesssim \|v_h\|_{a_h}$ by the definition of $a_h$ in \eqref{def:a-h}, \eqref{ineq:proof-conf-companion-b}, and finally $\|\varepsilon_\pw(\mathcal{R} v_h - \mathcal{J} v_h)\| \leq \|\D_\pw(\mathcal{R} v_h - \mathcal{J} v_h)\|$ conclude the proof.
\end{proof}

\begin{remark}[$\varepsilon_h = \Pi_{\Tcal}^k \circ \varepsilon \circ \mathcal{J}$]\label{rem:orthogonality-disc}
	The interplay of \Cref{lem:commuting-diagram} and the right-inverse $\mathcal{J}$ of $\I$ leads, for any $v_h \in V_h$, to $\Pi_\Tcal^k \varepsilon(\mathcal{J} v_h) = \varepsilon_h \I \mathcal{J} v_h = \varepsilon_h v_h$.
\end{remark}

\subsection{Quasi-best approximation of stabilization}
The quasi best-approximation $|\I v|_\s \lesssim \|\mu^{1/2}\varepsilon_\pw(v - \mathcal{R} \I v)\|$ is well-known in the literature \cite{DiPietroErn2015,ErnZanotti2020}.
Since the right-hand side of \eqref{ineq:a-priori} involves the $L^2$ projection, the following result is required. Recall the set $\mathbb{T}$ of admissible triangulations from \Cref{sec:triangulation}.
\begin{lemma}[quasi-best approximation of Galerkin projection]\label{lem:stability-Galerkin}
	There exist constants $\newcnst\label{cnst:stability-1}$, $\newcnst\label{cnst:stability-2}$, $\newcnst\label{cnst:stability-3}$, that exclusively depend on $k$ and $\mathbb{T}$, such that (a)--(b) hold.
	\begin{enumerate}
		\item[(a)] Any $v \in H^1(T)^n$ on $T \in \Tcal$ with $\Tcal \in \mathbb{T}$ satisfies
		\begin{align}\label{ineq:stability}
		\cnst{cnst:stability-1}^{-1}\s_T(\I v, \I v)^{1/2} \leq \|\varepsilon(v - \mathcal{R} \I v)\|_{L^2(T)} \leq \cnst{cnst:stability-2}\|(1 - \Pi_T^k) \varepsilon (v)\|_{L^2(T)}.
		\end{align}
		\item[(b)] Any $v \in V$ satisfies
		\begin{align}\label{ineq:stability-2}
			\|\varepsilon(v - \mathcal{J} \I v)\| \leq \cnst{cnst:stability-3} \|(1 - \Pi_\mathcal{T}^k) \varepsilon(v)\|_{L^2(\Omega)}.
		\end{align}
	\end{enumerate}
\end{lemma}
\begin{proof}[Proof of \Cref{lem:stability-Galerkin}.a]
	Since the first inequality is well-established, the focus is on the second one, which is an extension of the stability estimate in \cite[Theorem 3.1]{CarstensenZhaiZhang2020}. 
	The arguments therein are sketched below for completeness.
	Let
	\begin{align*}
		H^1(T)^n/\mathrm{RM} \coloneqq \{w \in H^1(T)^n : \int_T w \d{x} = 0 \text{ and } \int_T \D_\mathrm{ss} w \d{x} = 0\}.
	\end{align*}
	We define the Banach spaces $X \coloneqq \{v \in H^1(T)^n/\mathrm{RM}: v \perp P_{k+1}(T)^n\}$
	endowed with the scalar product $(\varepsilon \bullet, \varepsilon \bullet)_{L^2(T)}$;
	$Y \coloneqq L^2(T;\mathbb{S})$ and $Z \coloneqq P_k(T;\mathbb{S})$ endowed with the $L^2$ scalar product.
	Observe, for any $v \in H^1(T)^n$, that $\widetilde{v} \coloneqq v - \mathcal{R} \I v \in X$ from the convention \eqref{def:potential-reconstruction-2} and \eqref{ineq:ba-RI}. Since $(1 - \Pi_{T}^k) \varepsilon (\widetilde{v}) = (1 - \Pi_T^k) \varepsilon (v)$ from $\varepsilon (P_{k+1}(T)^n) \subset P_k(T;\mathbb{S})$, it suffices to show
	\begin{align}\label{ineq:stab-H1-X}
		\|\varepsilon (v)\|_{L^2(T)} \leq C_{\mathrm{stab}}(T)\|(1 - \Pi_T^k) \varepsilon (v)\|_{L^2(T)} \quad\text{for any } v \in X
	\end{align}
	with a positive constant $C_{\mathrm{stab}}(T) > 0$.
	The operators $A \coloneqq (1 - \Pi_T^k) \varepsilon : X \to Y$ and $B \coloneqq \Pi_T^k \varepsilon : X \to Z$ are linear and bounded. If $A v = 0$, then $\varepsilon (v) \in Z$. From \cite[Eq.~(3.16) in Chapter VI]{Braess2007}, we infer that
	$\D^2 v$ is a tensor-valued polynomial of degree at most $k-1$. Consequently, $v \in P_{k+1}(T)^n$, but the orthogonality $v \perp P_{k+1}(T)^n$ in $X$ reveals $v = 0$. 
	This implies the injectivity of $A$.
	Since $Z$ is a finite dimensional space, $B$ is a compact operator. The Pythagoras theorem provides
	\begin{align*}
		\|v\|_X^2 = \|\varepsilon(v)\|_{L^2(T)}^2 \leq \|(1 - \Pi_T^k) \varepsilon(v)\|_{L^2(T)}^2 + \|\Pi_T^k\varepsilon(v)\|_{L^2(T)}^2 = \|A v\|_Y^2 + \|B v\|_Z^2.
	\end{align*}
	Therefore, the Peetre--Tartar theorem \cite[Theorem 2.1]{GiraultRaviart1986} leads to \eqref{ineq:stab-H1-X} with a constant $C_\mathrm{stab}(T)$ that may depend on $T$ and $k$.
	It remains to prove that $C_\mathrm{stab}(T)$ is uniformly bounded for all $T \in \cup \mathbb{T}$.
	We recall that a congruent mapping $\Phi$ is of the form $\Phi(x) = \alpha(A x + b)$ with a positive number $\alpha > 0$, a vector $b \in \R^n$, and an orthonormal matrix $A$. It is well known \cite{Maubach1995,Stevenson2008} that the number of congruent classes in $\mathbb{T}$ is finite, i.e., there is a finite subset $\mathcal{M}$ of $\cup \mathbb{T}$ such that any $T \in \cup \mathbb{T}$ can be written as $T = \Phi(K)$ with some $K \in \mathcal{M}$ and congruent mapping $\Phi$.
	Theorem 4.1 from \cite{Maubach1995} shows that there exist finitely many (one-dimensional) lines through zero such that, given any simplex $T \in \mathbb{T}$ with ordered vertices $x_0, \dots, x_{n}$, the edges $x_1 - x_0$, $x_2 - x_0$, \dots, $x_n - x_0$ of $T$ lie on these lines. As a conclusion, it is possible to select finite many orthonormal matrices $A_1, \dots, A_J$ (rotations and reflections) such that any $T \in \cup\mathbb{T}$ satisfies $T = \Phi(K)$ for some $K \in \mathcal{M}$ with a congruent mapping $\Phi(x) = \alpha(A_j x + b)$ for some $\alpha > 0$, $b \in \R^n$, and $1 \leq j \leq J$. Set $\mathcal{S} \coloneqq \cup_{1 \leq j \leq J} A_j \mathcal{M}$, then any admissible simplex $T \in \cup \mathbb{T}$ is the image of some $K \in \mathcal{S}$ under translation and resizing. Notice that the constant $C_\mathrm{stab}(T)$ in \eqref{ineq:stab-H1-X} is invariant under these mappings. Since $\mathcal{S}$ is finite, $C_\mathrm{stab} \coloneqq \max_{T \in \mathcal{S}} C_\mathrm{stab}(T)$ is a uniform upper bound.
\end{proof}

\begin{proof}[Proof of \Cref{lem:stability-Galerkin}.b]
	In view of \eqref{ineq:stability}, it remains to prove
	\begin{align}\label{ineq:proof-stability}
		\|\varepsilon_\pw(\mathcal{J} \I v - \mathcal{R} \I v)\| \lesssim \|(1 - \Pi_{\Tcal}^k) \varepsilon(v)\|
	\end{align}
	thanks to a triangle inequality.
	\Cref{lem:conforming-companion} with $v_h = \I v$ and \Cref{lem:quasi-opt-R-I} provide
	\begin{align*}
		\|\varepsilon_\pw(\mathcal{J} \I v - \mathcal{R} \I v)\| \lesssim \|\D_\pw(v - \mathcal{R} \I v)\| + |\I v|_{\widehat{\s}} \lesssim \|\varepsilon_\pw(v - \mathcal{R} \I v)\| + |\I v|_{\widehat{\s}}.
	\end{align*}
	The combination of this with \eqref{ineq:stability} concludes the proof of \eqref{ineq:stability-2}.
\end{proof}

\subsection{tr-dev-div lemma}
We state a general version of the tr-dev-div lemma.
\begin{lemma}[tr-dev-div]\label{lem:tr-div-dev}
	Suppose (\ref{assumption}), then (a)--(d) hold.
	\begin{enumerate}[wide]
		\item[(a)] The subspace $\Sigma_0$ does not contain the identity matrix $\mathrm{I}_{n \times n} \notin \Sigma_0$.
		\item[(b)] Given $0 \leq s \leq 1$, any $\tau \in \Sigma_0$ satisfies
		\begin{align}\label{ineq:tr-dev-div}
			C_{\mathrm{tdd}}^{-1}\|\tr\,\tau\|_{H^s(\Omega)} \leq \|\mathrm{dev}\,\tau\|_{H^s(\Omega)} + \|\div\,\tau\|_{H^{s-1}(\Omega)}.
		\end{align}
		The constant $C_{\mathrm{tdd}}$ exclusively depends on $\Sigma_0$ and $s$.
		In particular,
		\begin{align}\label{ineq:L2-dev-div}
			\|\tau\|^2 \lesssim \|\mu^{1/2}\tau\|^2_{\mathbb{C}^{-1}} + \|\div\, \tau\|_{-1}^2.
		\end{align}
		\item[(c)] $\sigma - \sigma_h \in \Sigma_0$.
		\item[(d)] If $\psi \in V$ solves $(\varepsilon (\psi), \varepsilon (\varphi))_{\mathbb{C}} = (\sigma - \sigma_h, \varepsilon (\varphi))_{L^2(\Omega)}$ for any $\varphi \in V$, then $\mathbb{C} \varepsilon(\psi) \in \Sigma_0$.
	\end{enumerate}
\end{lemma}
\begin{proof}[Proof of \Cref{lem:tr-div-dev}.a]
	The case $|\Gamma_{\mathrm{N}}| = 0$ is trivial. If $|\Gamma_{\mathrm{N}}| > 0$, then the function $\varphi_0 \in S^k_\mathrm{D}(\Tcal_0)$ from \eqref{assumption} satisfies $\int_\Omega \I_{n \times n} : \varepsilon(\varphi_k) \d{x} = \int_\Omega \div \varphi_k \d{x} \neq 0$. Hence, $\I_{n \times n} \notin \Sigma_0$.\vspace*{0.2cm}
	
	\noindent\emph{Proof of \Cref{lem:tr-div-dev}.b.}
	There are several variants of the tr-dev-div lemma known in the literature for $s = 0$, e.g., \cite[Proposition 9.1.1]{BoffiBrezziFortin2013}.
	This version \eqref{ineq:tr-dev-div} with fractional-order Sobolev norms is recently established in \cite[Theorem 1]{CarstensenHeuer2024}.
	The bound \eqref{ineq:L2-dev-div} follows from \eqref{ineq:tr-dev-div} with $s = 0$ and the algebraic inequality $|\tau|^2 = |\mathrm{dev} \tau|^2 + |\tr \tau|^2/n \leq 2\mu|\mathbb{C}^{-1/2}\tau|^2 + |\tr \tau|^2/n$ pointwise a.e.~in $\Omega$.\vspace*{0.2cm}
	
	\noindent\emph{Proof of \Cref{lem:tr-div-dev}.c.}
	Let $|\Gamma_{\mathrm{N}}| = 0$. Recall that $\lambda$ and $\mu$ are constant.
	The choice $\Phi_k \coloneqq \I_{n \times n}$ in \eqref{def:sym-grad-rec} and $u_\Fcal = 0$ along $\Gamma_{\mathrm{D}} = \partial \Omega$ provide
	\begin{align*}
		(\varepsilon_h u_h, \I_{n \times n})_{L^2(\Omega)} = \sum\nolimits_{F \in \Fcal(\partial \Omega)} \int_F u_F \cdot \nu_F \d{s} = 0. 
	\end{align*}
	This and an integration by parts show
	\begin{align}\label{eq:tr-sigma-sigma-h}
		(2\mu + \lambda)^{-1}\int_\Omega \tr (\sigma - \sigma_h) \d{x} &= \int_\Omega \div\, u \d{x} - (\varepsilon_h u_h, \I_{n \times n})_{L^2(\Omega)}\nonumber\\
		&= \int_{\Gamma_\mathrm{D}} (u - u_\Fcal) \cdot \nu \d{s} = 0,
	\end{align}
	whence $\sigma - \sigma_h \in \Sigma_0$.
	Consider the case $|\Gamma_{\mathrm{N}}| > 0$.
	Recall $\s(u_h, \I \varphi_0) = 0$ from \Cref{cor:kernel-stabilization} and $\varepsilon_h \I \varphi_0 = \varepsilon(\varphi_0)$ from \Cref{rem:RI=Id} with $\varphi_{0} \in S^{k}_\mathrm{D}(\Tcal_0) \subset S^{k}_\mathrm{D}(\Tcal)$ from \eqref{assumption}.
	The variational formulations \eqref{def:continuous-problem} and \eqref{def:discrete-problem}  provide
	\begin{align}\label{eq:proof-reliability-sigma-sigma-h-constant-free}
		(\sigma - \sigma_h, \varepsilon (\varphi_0))_{L^2(\Omega)} = (\sigma, \varepsilon(\varphi_0))_{L^2(\Omega)} - (\sigma_h, \varepsilon_h \I \varphi_0)_{L^2(\Omega)} = 0.
	\end{align}
	
	\noindent\emph{Proof of \Cref{lem:tr-div-dev}.d.}
	If $|\Gamma_{\mathrm{N}}| = 0$, then $\int_\Omega \tr\, \mathbb{C} \varepsilon(\psi) \d{x} = (2\mu + \lambda)\int_\Omega \div\, \psi \d{x} = 0$ from an integration by parts formula.
	If $|\Gamma_{\mathrm{N}}| > 0$, then the $L^2$ orthogonality $\mathbb{C} \varepsilon(\psi) \perp \varepsilon(S_\mathrm{D}^k(\Tcal))$ by design of $\psi$ and $\sigma - \sigma_h \in \Sigma_0$ from (c) show $\mathbb{C} \varepsilon(\psi) \in \Sigma_0$.
\end{proof}

We note that $\|\div(\sigma - \sigma_h)\|_{-1}$ arising in \Cref{lem:tr-div-dev}.b can be bounded as follows.

\begin{lemma}[bound $\|\div(\sigma - \sigma_h)\|_{-1}$]\label{lem:divergence-stress}
	The discrete stress $\sigma_h$ satisfies
	\begin{align*}
		\|\div(\sigma - \sigma_h)\|_{-1} \lesssim \mathrm{osc}(f,\Tcal) + \mu_1^{1/2}|u_h|_\s.
	\end{align*}
\end{lemma}
\begin{proof}
	Given any $\varphi \in V$,
	\Cref{lem:commuting-diagram} and \eqref{def:discrete-problem} imply
	\begin{align*}
		(\sigma_h, \varepsilon(\varphi))_{L^2(\Omega)} = (\sigma_h, \varepsilon_h \I \varphi)_{L^2(\Omega)} = (f, \Pi_{\Tcal}^k \varphi)_{L^2(\Omega)} + (g, \Pi_\Fcal^k \varphi)_{L^2(\Gamma_{\mathrm{N}})} - \s(u_h,\I \varphi).
	\end{align*}
	This and the variational formulation \eqref{def:continuous-problem} provide
	\begin{align}\label{ineq:data-osc}
		(\sigma - \sigma_h, \varepsilon(\varphi))_{L^2(\Omega)} &= (f,(1 - \Pi_{\Tcal}^k)\varphi)_{L^2(\Omega)}\nonumber\\
		&\qquad\qquad + (g, (1 - \Pi_\Fcal^k) \varphi)_{L^2(\Gamma_{\mathrm{N}})} + \s(u_h,\I \varphi).
	\end{align}
	Standard arguments with Cauchy and Poincar\'e inequalities lead to
	\begin{align}
		(\sigma - \sigma_h, \varepsilon(\varphi))_{L^2(\Omega)} \lesssim \mathrm{osc}(f,\Tcal)\|\D \varphi\| + |u_h|_\s |\I \varphi|_\s
	\end{align}
	for any $\varphi \in H^1_0(\Omega)^n$.
	Since $|\I \varphi|_\s \lesssim \mu_1^{1/2}\|\varepsilon(\varphi)\| \leq \mu_1^{1/2}\|\D \varphi\|$ from \eqref{ineq:stability}, this and the definition \eqref{def:dual-norm} of $\|\bullet\|_{-1}$ conclude the assertion.
\end{proof}

\subsection{Proof of \Cref{thm:best-approximation}}\label{sec:proof-a-priori}
The proof is divided into two parts. The first part establishes the quasi-best approximation estimate
\begin{align}\label{ineq:ba-discrete-bilinear-form}
	\|e_h\|_{a_h} + |u_h|_\s \lesssim \mu_0^{-1/2}(\|(1 - \Pi_{\Tcal}^k) \sigma\| + \mathrm{osc}(f,\mathcal{T}) + \mathrm{osc}(g,\mathcal{F}_\mathrm{N}))
\end{align}
for the discrete error $\|e_h\|_{a_h}$ with the abbreviation $e_h = (e_\Tcal, e_\Fcal) \coloneqq \I u - u_h \in V_h$, while the second part controls the $L^2$ error of $\sigma - \sigma_h$, namely
\begin{align}\label{ineq:ba-L2-error}
	\|\sigma - \sigma_h\| \lesssim (\mu_1/\mu_0)^{1/2}(\|(1 - \Pi_{\Tcal}^k) \sigma\| + \mathrm{osc}(f,\mathcal{T}) + \mathrm{osc}(g,\mathcal{F}_\mathrm{N})).
\end{align}
\noindent 3.5.1. \emph{Proof of \eqref{ineq:ba-discrete-bilinear-form}.}
From the identity
$- 2\s(u_h, e_h) = |e_h|_\s^2 + |u_h|_\s^2 - |\I u|_\s^2$, we deduce that
\begin{align}\label{eq:proof-best-approx-split}
	\|\varepsilon_h e_h\|_\mathbb{C}^2 &+ (|e_h|_\s^2 + |u_h|_\s^2 - |\I u|_\s^2)/2\nonumber\\
	& = \|e_h\|_{a_h}^2 - \s(\I u, e_h) = (\varepsilon_h\I u, \varepsilon_h e_h)_{\mathbb{C}} - a_h(u_h, e_h).
\end{align}
The formula $\I \circ \mathcal{J} = \mathrm{Id}$ from \Cref{lem:conforming-companion} provides $(\varepsilon_h \I u, \varepsilon_h e_h)_{\mathbb{C}} = (\varepsilon_h \I u, \varepsilon (\mathcal{J} e_h))_{\mathbb{C}}$.
This, $\sigma = \mathbb{C} \varepsilon(u)$, and $\Pi_\Tcal^k \sigma = \mathbb{C} \varepsilon_h \I u$ from \Cref{lem:commuting-diagram}
imply
\begin{align}
	\label{eq:proof-best-approximation-1}
	(\varepsilon_h \I u, \varepsilon_h e_h)_{\mathbb{C}} = (\sigma, \varepsilon (\mathcal{J} e_h))_{L^2(\Omega)}
	- ((1 - \Pi_\Tcal^k) \sigma, \varepsilon (\mathcal{J} e_h))_{L^2(\Omega)}.
\end{align}
Observe $(\sigma_h, \varepsilon_h e_h)_{L^2(\Omega)} = (\sigma_h, \varepsilon(\mathcal{J} e_h))_{L^2(\Omega)}$ from \Cref{rem:orthogonality-disc} to verify
\begin{align*}
	(\sigma, \varepsilon (\mathcal{J} e_h))_{L^2(\Omega)} - a_h(u_h, e_h) = (\sigma - \sigma_h, \varepsilon(\mathcal{J} e_h))_{L^2(\Omega)} - \s(u_h,e_h).
\end{align*}
The right-hand side of this is given by \eqref{ineq:data-osc} with $\varphi \coloneqq \mathcal{J} e_h$. Thus, standard arguments with Cauchy, trace, Poincar\'e, and Korn inequalities provide
\begin{align}\label{ineq:proof-ba-osc}
	(\sigma, \varepsilon (\mathcal{J} e_h))_{L^2(\Omega)} - a_h(u_h, e_h) \lesssim (\mathrm{osc}(f,\Tcal) + \mathrm{osc}(g,\Fcal_\mathrm{N}))\|\varepsilon(\mathcal{J} e_h)\|.
\end{align}
Since $\mu_0^{1/2}\|\varepsilon (\mathcal{J} e_h)\| \lesssim \|e_h\|_{a_h}$ from \Cref{lem:conforming-companion}.b,
the combination of \eqref{eq:proof-best-approx-split}--\eqref{ineq:proof-ba-osc} results in
\begin{align}\label{ineq:proof-ba-step-1}
	&\|\varepsilon_h e_h\|_\mathbb{C} + |e_h|_\s + |u_h|_\s\nonumber\\
	&\qquad \lesssim \mu_0^{-1/2}\|(1 - \Pi_{\Tcal}^k) \sigma\| + |\I u|_\s + \mu_0^{-1/2}\mathrm{osc}(f,\Tcal) + \mu_0^{-1/2}\mathrm{osc}(g,\Fcal_\mathrm{N}).
\end{align}
From \Cref{lem:stability-Galerkin}, we infer
\begin{align*}
	|\I u|_{\s} &\lesssim \|\mu^{1/2}(1 - \Pi_\Tcal^k) \varepsilon(u)\| \lesssim \mu_0^{-1/2} \|(1 - \Pi_\Tcal^k) \sigma\|.
\end{align*}
This and \eqref{ineq:proof-ba-step-1} conclude the proof of \eqref{ineq:ba-discrete-bilinear-form}.\\[-0.5em]

\noindent 3.5.2. \emph{Proof of \eqref{ineq:ba-L2-error}.}
Recall from \Cref{lem:tr-div-dev}.c that $\sigma - \sigma_h \in \Sigma_0$. In view of \eqref{ineq:L2-dev-div}, \Cref{lem:divergence-stress}, and \eqref{ineq:ba-discrete-bilinear-form}, it remains to prove
\begin{align}\label{ineq:ba-sigma-simgah-C-1}
	\|\mu^{1/2}(\sigma - \sigma_h)\|_{\mathbb{C}^{-1}} \lesssim (\mu_1/\mu_0)^{1/2}(\|(1 - \Pi_{\Tcal}^k) \sigma\| + \mathrm{osc}(f,\Tcal) + \mathrm{osc}(g, \Fcal_\mathrm{N}))
\end{align}
for \eqref{ineq:ba-L2-error}.
The Pythagoras theorem with the $L^2$ orthogonality $\mathbb{C}^{-1/2}(\sigma - \mathbb{C} \varepsilon_h \I u) \perp P_k(\mathcal{T})^{n \times n}$ from
\Cref{lem:commuting-diagram} 
proves
\begin{align}\label{ineq:proof-ba-sigma-sigma-h-C-1}
	\|\mu^{1/2}(\sigma - \sigma_h)\|^2_{\mathbb{C}^{-1}} &=  \|\mu^{1/2}(\sigma - \mathbb{C} \varepsilon_h \I u)\|^2_{\mathbb{C}^{-1}} + \|\mu^{1/2}\varepsilon_h e_h\|^2_{\mathbb{C}}\nonumber\\
	&\lesssim \|(1 - \Pi_{T}^k) \sigma\|^2 + \|\mu^{1/2}\varepsilon_h e_h\|^2_{\mathbb{C}}.
\end{align}
This, the bound $\|\varepsilon_h e_h\|_\mathbb{C} \leq \|e_h\|_{a_h}$ from \eqref{def:a-h}, and \eqref{ineq:ba-discrete-bilinear-form} imply \eqref{ineq:ba-sigma-simgah-C-1}. Therefore,
we have established \eqref{ineq:ba-L2-error}.
The combination of \eqref{ineq:ba-discrete-bilinear-form}--\eqref{ineq:ba-L2-error} concludes the proof of \Cref{thm:best-approximation}.\hfill\qed

\subsection{$L^2$ error estimate}
Elliptic regularity allows for $L^2$ error estimates with higher convergence rates in comparison to \Cref{thm:best-approximation}.
We assume that $s > 0$ (index of elliptic regularity) satisfies the following: The solution $z \in H^1_0(\Omega)^n$ to
\begin{align}\label{def:aux-problem}
	-\div\, \mathbb{C} \varepsilon(z) = f \text{ in } \Omega
\end{align}
for any $f \in L^2(\Omega)^n$
satisfies $z \in H^{1+s}(\Omega)^n$ with a regularity index $s > 0$ (depending on the polyhedral domain $\Omega$) and
\begin{align}\label{ineq:elliptic-regularity}
	\|\mu\varepsilon(z)\|_{H^{s}(\Omega)} + \|\lambda\div\,z\|_{H^s(\Omega)} \lesssim c_{\mu,s}\|f\|
\end{align}
with a constant $c_{\mu,s}$ depending on $\mu$ and $s$.
Such regularity results are available for constant parameters $\lambda$ and $\mu$ under additional assumptions. For instance, \eqref{ineq:elliptic-regularity} is established in \cite{BrennerSung1992} with $s = 1$ on convex planar domains.
Singular functions for general polygonal domains with mixed boundary conditions and $g \equiv 0$ have been computed in \cite{Roessle2000} and lead to $z \in H^{1+s}(\Omega)^2$ for some explicit $s > 0$ as well as $\|z\|_{H^{1+s}(\Omega)} \lesssim \|f\|$. Thereafter, \Cref{lem:tr-div-dev} controls the norm $\lambda\|\div z\|_{H^s(\Omega)}$ and provides \eqref{ineq:elliptic-regularity}.

While convergence rates of the $L^2$ error have been derived in \cite[Subsection 5.2]{DiPietroErn2015}, the emphasis here is on quasi-best approximation error estimates.

\begin{theorem}[$L^2$ error estimate]\label{thm:L2-estimate}
	If \eqref{assumption} and \eqref{ineq:elliptic-regularity}, then
	\begin{align}
		\|\Pi_{\Tcal}^k u - u_\Tcal\| \lesssim c_{\mu,s}\mu_0^{-1}h_\mathrm{max}^s(\|(1 - \Pi_\Tcal^k) \sigma\| + \mathrm{osc}(f,\Tcal) + \mathrm{osc}(g,\Fcal_\mathrm{N})).
	\end{align} 
\end{theorem}
\begin{proof}
	Abbreviate $e_h = (e_\Tcal, e_\Fcal) \coloneqq \I u - u_h \in V_h$. Let $z \in V$ solve \eqref{def:aux-problem} with the right-hand side $e_\Tcal$. A consequence of 
	\eqref{ineq:elliptic-regularity} and piecewise Poincar\'e inequalities is
	\begin{align}\label{ineq:elliptic-regularity-2}
		\|(1 - \Pi_{\Tcal}^k) \mathbb{C}\varepsilon(z)\| \leq 2\|\mu(1 - \Pi_{\Tcal}^k) \varepsilon(z)\| + \|\lambda(1 - \Pi_{\Tcal}^k) \div z\| \lesssim c_{\mu,s}h_\mathrm{max}^s \|e_\Tcal\|.
	\end{align}
	The $L^2$ orthogonality $\mathcal{J} e_h - e_\Tcal \perp e_\Tcal$ and the solution property of $z$ imply
	\begin{align}\label{eq:proof-L2-split}
		\|e_\Tcal\|^2 &= (e_\Tcal, \mathcal{J} e_h)_{L^2(\Omega)} = (\mathbb{C} \varepsilon(z), \varepsilon(\mathcal{J} e_h))_{L^2(\Omega)}\nonumber\\
		&= (\mathbb{C} \varepsilon(z), \varepsilon(u) - \varepsilon_h u_h)_{L^2(\Omega)} - (\mathbb{C} \varepsilon(z), \varepsilon(u) - \varepsilon_h\I u + \varepsilon_h e_h - \varepsilon(\mathcal{J} e_h))_{L^2(\Omega)}.
	\end{align}
	Since $\sigma - \sigma_h = \mathbb{C}(\varepsilon(u) - \varepsilon_h(u_h))$,
	the first term can be rewritten as $(\mathbb{C} \varepsilon(z), \varepsilon(u) - \varepsilon_h u_h)_{L^2(\Omega)} = (\sigma - \sigma_h, \varepsilon(z))_{L^2(\Omega)}$.
	This is the left-hand side of \eqref{ineq:data-osc} with $\varphi \coloneqq z$.
	Cauchy, Poincar\'e, and trace inequalities lead to
	\begin{align*}
		(\mathbb{C} \varepsilon(z), \varepsilon(u) - \varepsilon_h u_h)_{L^2(\Omega)} \lesssim (\mathrm{osc}(f,\Tcal) + \mathrm{osc}(g,\Fcal_\mathrm{N}))\|\D_\pw(1 - \Pi_{\Tcal}^k) z\| + |u_h|_\s|\I z|_\s.
	\end{align*}
	This, $|\I z|_\s \lesssim \|\mu^{1/2}(1 - \Pi_{\Tcal}^k)\varepsilon(z)\|$ from \Cref{lem:stability-Galerkin}, the regularity $z \in H^{1+s}(\Omega)^n$ from \eqref{ineq:elliptic-regularity}, and $k \geq 1$ provide
	\begin{align}\label{ineq:proof-L2-T1}
		 &(\mathbb{C} \varepsilon(z), \varepsilon(u) - \varepsilon_h u_h)_{L^2(\Omega)}\nonumber\\
		 &\qquad\lesssim c_{\mu,s}\mu_0^{-1} h_\mathrm{max}^s(\mathrm{osc}(f,\Tcal) + \mathrm{osc}(g,\Fcal_\mathrm{N}) + \mu_0^{1/2}|u_h|_\s)\|e_\Tcal\|.
	\end{align}
	Recall the $L^2$ orthogonality $\varepsilon(u) - \varepsilon_h \I u \perp P_k(\mathcal{T})^{n \times n}$ from \Cref{lem:commuting-diagram} and $\varepsilon_h e_h - \varepsilon(\mathcal{J} e_h) \perp P_k(\mathcal{T})^{n \times n}$ from \Cref{rem:orthogonality-disc}. The Cauchy inequality and \eqref{ineq:elliptic-regularity-2} imply
	\begin{align*}
		- (\mathbb{C} \varepsilon(z)&, \varepsilon(u) - \varepsilon_h\I u + \varepsilon_h e_h - \varepsilon(\mathcal{J} e_h))_{L^2(\Omega)}\\
		&\leq \|(1 - \Pi_{\Tcal}^k)\mathbb{C} \varepsilon(z)\|(\|(1 - \Pi_\Tcal^k) \varepsilon(u)\| + \|\varepsilon_h e_h - \varepsilon(\mathcal{J} e_h)\|)\\
		& \lesssim c_{\mu,s}h_\mathrm{max}^s(\|(1 - \Pi_\Tcal^k) \varepsilon(u)\| + \|\varepsilon_h e_h  - \varepsilon(\mathcal{J} e_h)\|)\|e_\Tcal\|.
	\end{align*}
	This, the stability $\|\varepsilon_h e_h  - \varepsilon(\mathcal{J} e_h)\| \lesssim \mu_0^{-1/2}\|e_h\|_{a_h}$ from \Cref{lem:conforming-companion}, and $2\mu_0\|(1 - \Pi_\Tcal^k) \varepsilon(u)\| \leq \|(1 - \Pi_\Tcal^k) \sigma\|$ establish
	\begin{align}\label{ineq:proof-L2-T2}
		- (\mathbb{C} \varepsilon(z)&, \varepsilon(u) - \varepsilon_h\I u + \varepsilon_h e_h - \varepsilon(\mathcal{J} e_h))_{L^2(\Omega)}\nonumber\\
		& \lesssim
		c_{\mu,s}\mu_0^{-1}h_\mathrm{max}^s(\|(1 - \Pi_{\Tcal}^k) \sigma\| + \mu_0^{1/2}\|e_h\|_{a_h})\|e_\Tcal\|.
	\end{align}
	The combination of \eqref{eq:proof-L2-split}--\eqref{ineq:proof-L2-T2} with \eqref{ineq:a-priori} concludes the proof.
\end{proof}

\subsection{Quasi-best approximation with smoother}
If the given data $f \in V^*$ or $g \in \widetilde{H}^{-1/2}(\Gamma_{\mathrm{N}})^n$ is non-smooth, then the operator $\mathcal{J}$ can be utilized in an alternative HHO method in the spirit of \cite{VeeserZanottiII,ErnZanotti2020}.
The continuous problem seeks $u \in V$ such that any $v \in V$ satisfies
\begin{align}\label{def:continuous-problem-2}
	(\mathbb{C} \varepsilon(u), \varepsilon(v))_{L^2(\Omega)} = \langle f, v\rangle_{V^* \times V} + \langle g, v\rangle_{\widetilde{H}^{-1/2}(\Gamma_{\mathrm{N}}) \times \widetilde{H}^{1/2}(\Gamma_{\mathrm{N}})}.
\end{align}
The corresponding discrete problem
seeks $u_h \in V_h$ with
\begin{align}\label{def:discrete-problem-2}
	a_h(u_h,v_h) = \langle f, \mathcal{J} v_h\rangle_{V^* \times V} + \langle g, \mathcal{J} v_h\rangle_{\widetilde{H}^{-1/2}(\Gamma_{\mathrm{N}}) \times \widetilde{H}^{1/2}(\Gamma_{\mathrm{N}})}
\end{align}
for any $v_h \in V_h$.
Let $\sigma = \mathbb{C}\varepsilon(u)$ and $\sigma_h = \mathbb{C} \varepsilon_h u_h$ denote the continuous and discrete stress variable.
The arguments of this section imply the following oscillation-free and $\lambda$-robust quasi-best approximation error estimate.

\begin{theorem}[a~priori with smoother]
	If \eqref{assumption} and \eqref{ineq:elliptic-regularity} hold, then
	\begin{align}\label{ineq:ba-smoother}
		(\mu_0/\mu_1)^{1/2}\|\sigma - \sigma_h\| + \mu_0^{1/2}\|\I u - u_h\|_{a_h} + \mu_0^{1/2}|u_h|_\s + c_{\mu,s}^{-1}\mu_0 h_{\max}^{-s}\|\Pi_{\Tcal}^k u - u_\Tcal\|&\nonumber\\
		\lesssim \|(1 - \Pi_{\Tcal}^k) \sigma\|&.
	\end{align}
\end{theorem}
\begin{proof}
	The arguments in the proof of \Cref{thm:best-approximation} apply to this case as well with the following adjustments.
	Abbreviate $e_h \coloneqq \I u - u_h \in V_h$.
	First, observe from the variational formulations \eqref{def:continuous-problem-2}--\eqref{def:discrete-problem-2} that
	\begin{align*}
		(\sigma, \varepsilon(\mathcal{J} e_h))_{L^2(\Omega)} - a_h(u_h,e_h) = 0.
	\end{align*}
	Therefore, the data oscillations in \eqref{ineq:proof-ba-osc} do not arise and
	\begin{align}\label{ineq:proof-ba-smoother-I}
		\|e_h\|_{a_h} + |u_h|_\s \lesssim \mu_0^{-1/2}\|(1 - \Pi_{\Tcal}^k) \sigma\|
	\end{align}
	holds instead of \eqref{ineq:proof-ba-step-1}. 
	The second observation is the projection property
	\begin{align*}
		\Pi_{\Tcal}^k \varepsilon(\mathcal{J} \I \varphi) = \varepsilon_h \I \varphi = \Pi_{\Tcal}^k \varepsilon(\varphi)
	\end{align*}
	for all $\varphi \in V$ from \Cref{lem:commuting-diagram} and
	$\I\circ \mathcal{J} = \mathrm{Id}$ by design of the right-inverse $\mathcal{J}$.
	This, $\sigma_h \in P_k(\mathcal{T})^{n \times n}$, and
	the variational formulations \eqref{def:continuous-problem-2}--\eqref{def:discrete-problem-2} reveal
	\begin{align}\label{eq:proof-ba-smoother-1}
		(\sigma - \sigma_h, \varepsilon(\mathcal{J} \I \varphi))_{L^2(\Omega)} = (\sigma, \varepsilon(\mathcal{J} \I \varphi))_{L^2(\Omega)} - (\sigma_h, \varepsilon_h \I \varphi)_{L^2(\Omega)} = - \s(u_h,\I \varphi).
	\end{align}
	Recall $|\I \varphi|_\s \lesssim \mu_1^{1/2}\|\varepsilon(\varphi)\|$ from \Cref{lem:equiv-stab} and $\mu_0^{1/2}|u_h|_\s \lesssim \|(1 - \Pi_{\Tcal}^k) \sigma\|$ from \eqref{ineq:proof-ba-smoother-I}. Therefore, a Cauchy inequality in \eqref{eq:proof-ba-smoother-1} provides
	\begin{align*}
		(\sigma - \sigma_h, \varepsilon(\mathcal{J} \I \varphi))_{L^2(\Omega)} \lesssim (\mu_1/\mu_0)^{1/2}\|(1 - \Pi_{\Tcal}^k) \sigma\| \|\varepsilon(\varphi)\|.
	\end{align*}
	The combination of this with $\varepsilon(\varphi - \mathcal{J} \I \varphi) \perp P_k(\mathcal{T})^{n \times n}$, $\|\varepsilon(\varphi - \mathcal{J} \I \varphi)\| \lesssim \|\varepsilon(\varphi)\|$ from \eqref{ineq:stability-2}, and a Cauchy inequality result in 
	\begin{align*}
		(\sigma - \sigma_h, \varepsilon(\varphi))_{L^2(\Omega)} &= ((1 - \Pi_\Tcal^k) \sigma, \varepsilon(\varphi - \mathcal{J} \I \varphi))_{L^2(\Omega)} - (\sigma - \sigma_h, \varepsilon(\mathcal{J} \I \varphi))\\
		& \lesssim (\mu_1/\mu_0)^{1/2}\|(1 - \Pi_{\Tcal}^k) \sigma\|\|\varepsilon(\varphi)\|.
	\end{align*}
	The supremum of this over all $\varphi \in H^1_0(\Omega)^n$ with the normalization $\|\D \varphi\| = 1$ provides $\|\div(\sigma - \sigma_h)\|_{-1} \lesssim (\mu_1/\mu_0)^{1/2}\|(1 - \Pi_{\Tcal}^k) \sigma\|$. This, \eqref{ineq:proof-ba-sigma-sigma-h-C-1}, and $\|\varepsilon_h e_h\|_{\mathbb{C}} \leq \|e_h\|_{a_h} \lesssim \mu_0^{-1/2}\|(1 - \Pi_\Tcal^k) \sigma\|$ from \eqref{ineq:proof-ba-smoother-I} imply 
	$$\|\mu^{1/2}(\sigma - \sigma_h)\|_{\mathbb{C}^{-1}}^2 \lesssim (\mu_1/\mu_0)\|(1 - \Pi_{\Tcal}^k) \sigma\|^2.$$
	For the modified HHO method \eqref{def:discrete-problem-2}, \eqref{eq:proof-reliability-sigma-sigma-h-constant-free} holds verbatim and this leads to $\sigma - \sigma_h \in \Sigma_0$ as in \Cref{lem:tr-div-dev}.c.
	In view of \eqref{ineq:L2-dev-div}, we have proven $\|\sigma - \sigma_h\| \lesssim (\mu_1/\mu_0)^{1/2}\|(1 - \Pi_{\Tcal}^k) \sigma\|$. This and \eqref{ineq:proof-ba-smoother-I} conclude the proof of
	\begin{align}\label{ineq:proof-a-priori-smoother}
		(\mu_0/\mu_1)^{1/2}\|\sigma - \sigma_h\| + \mu_0^{1/2}\|e_h\|_{a_h} + \mu_0^{1/2}|u_h|_\s
		\lesssim \|(1 - \Pi_{\Tcal}^k) \sigma\|.
	\end{align}
	To derive $L^2$ estimates, adopt the notation of $z$ and $e_h$ from the proof of \Cref{thm:L2-estimate}. Recall that $(\mathbb{C} \varepsilon(z), \varepsilon(u) - \varepsilon_h u_h)_{L^2(\Omega)} = (\sigma - \sigma_h, \varepsilon(z))_{L^2(\Omega)} \leq |u_h|_\s |\I z|_\s$ from \eqref{eq:proof-ba-smoother-1}. The regularity \eqref{ineq:elliptic-regularity} and \Cref{lem:stability-Galerkin}.a allow for $|\I z|_\s \lesssim \|\mu^{1/2}(1 - \Pi_{\Tcal}^k) \varepsilon(z)\| \lesssim c_{\mu,s}\mu_0^{-1/2}h_{\max}^s\|e_\Tcal\|$. Therefore,
	\begin{align}\label{ineq:proof-L2-smoother-1}
		(\mathbb{C} \varepsilon(z), \varepsilon(u) - \varepsilon_h u_h)_{L^2(\Omega)} \lesssim c_{\mu,s}\mu_0^{-1/2}h_{\max}^s|u_h|_\s \|e_\Tcal\|.
	\end{align}
	Since \eqref{eq:proof-L2-split} and \eqref{ineq:proof-L2-T2} only involve the orthogonality arising from \Cref{lem:commuting-diagram} and \Cref{lem:conforming-companion}, they hold verbatim. The combination of this with \eqref{ineq:proof-a-priori-smoother}--\eqref{ineq:proof-L2-smoother-1} reveals $\|\Pi_{\Tcal}^k u - u_\Tcal\| \lesssim c_{\mu,s}\mu_0^{-1} h_{\max}^s\|(1 - \Pi_{\Tcal}^k) \sigma\|$.
\end{proof}

\section{Stabilization-free and $\lambda$-robust a posteriori error analysis}\label{sec:a-posteriori}

This section derives the $\lambda$-robust reliable and efficient error estimate $\eta$ from \eqref{def:eta}.

\subsection{Main result}
The main result of this section establishes the a~posteriori error estimate \eqref{ineq:reliability} for the discretization \eqref{def:discrete-problem} without smoother.

\begin{theorem}[a~posteriori]\label{thm:reliability}
	Suppose \eqref{assumption}, then \eqref{ineq:reliability} holds with
	$\lambda$-independent constants $C_\mathrm{rel}$ and $C_\mathrm{eff}$ that exclusively depend on $k$, $\Gamma_{\mathrm{D}}$, $\Sigma_0$, and $\mathbb{T}$.
\end{theorem}

\begin{remark}[a~posteriori with smoother]
	For the discretization \eqref{def:discrete-problem-2} with smoother and $L^2$ right-hand sides, \Cref{thm:reliability} holds verbatim. The situation is, unfortunately, much more involved for $H^{-1}$ right-hand sides, where explicit error control relies on additional a~priori information on $f$. We omit further comments on such cases, but refer to \cite{KreuzerVeeser2021} for an abstract approach and \cite{CarstensenGraessleNataraj2024} for explicit examples.
\end{remark}

\begin{remark}[Comparison with \cite{BertrandCarstensenGraessle2021}]
	The only other stabilisation-free a posteriori error analysis for some HHO method applied to second-order PDE is \cite{BertrandCarstensenGraessle2021} on the Poisson model problem with an orthogonality of the piecewise gradients of some representation $\mathcal{R} u_h$ of the HHO displacement  $u_h$ to
	divergence-free lowest-order Raviart-Thomas finite element functions.
	It is however difficult to mimic the Helmholtz decomposition from \cite{BertrandCarstensenGraessle2021}, which is in particular challenging in 3D and in presence of mixed boundary conditions.
\end{remark}

\begin{remark}[other applications]
	The a~posteriori error analysis of this section can be extended to the scalar elliptic PDE 
	$-\div A \nabla u = f$ in $\Omega$, $u = 0$ on $\Gamma_{\mathrm{D}}$, $\sigma \nu = g$ on $\Gamma_{\mathrm{N}}$ with the stress $\sigma = A \nabla u$ and pointwise symmetric positive definite coefficients $A \in P_0(\mathcal{T})^{n \times n}$. Let the discrete solution $u_h \in V_h$ solve
	\begin{align*}
		\int_\Omega A \mathcal{G} u_h \cdot \mathcal{G}_h v_h = \int_\Omega f v_\Tcal \d{x} + \int_{\Gamma_\mathrm{N}} g v_\Fcal \d{s} \quad\text{for any } v_h = (v_\Tcal, v_\Fcal) \in V_h,
	\end{align*}
	where the notation of this paper is carried over to the scalar case. Then our analysis below establishes that the discrete stress $\sigma_h \coloneqq A \mathcal{G}_h u_h$ satisfies
	\begin{align*}
		\|\sigma - \sigma_h\|_{A^{-1}}^2 &\lesssim \|h_\Tcal(f + \div_\pw \sigma_h)\|^2 + \min_{v \in V} \|\nabla v - \mathcal{G}_h u_h\|^2_A\\
		&\qquad+ \sum_{F \in \Fcal(\Omega)} h_F\|[\sigma_h]_F \nu_F\|_{L^2(F)}^2 +
		\sum_{F \in \Fcal_\mathrm{N}} h_F\|g - \sigma_h \nu_F\|_{L^2(F)}^2\\
		&\lesssim \|\sigma - \sigma_h\|_{A^{-1}}^2 + \mathrm{osc}^2(f,\Tcal) + \mathrm{osc}^2(g,\Fcal_\mathrm{N})
	\end{align*}
	with the weighted norm $\|\bullet\|_{A^s}^2 \coloneqq (A^s \bullet, \bullet)_{L^2(\Omega)}$ for $s = \pm 1$.
\end{remark}

\subsection{Proof of reliability}\label{sec:proof-reliability}
The proof is divided into three steps.\\[-0.5em]

\noindent4.2.1. \emph{Orthogonal split.}
Recall $\psi \in V$ from \Cref{lem:tr-div-dev}.d
and define the divergence free function $\tau \coloneqq \sigma - \sigma_h - \mathbb{C} \varepsilon (\psi) \in L^2(\Omega)^{n \times n}$ with the $L^2$ orthogonality $\tau \perp \varepsilon(V)$.
The latter is equivalent to the $L^2$ orthogonality $\mathbb{C}^{-1/2} \tau \perp \mathbb{C}^{1/2} \varepsilon(V)$ and the Pythagoras theorem provides the split
\begin{align}\label{eq:proof-reliability-split}
	\|\sigma - \sigma_h\|_{\mathbb{C}^{-1}}^2 &= \|\varepsilon (\psi)\|_\mathbb{C}^2 + \|\tau\|_{\mathbb{C}^{-1}}^2.
\end{align}
\noindent4.2.2. \emph{Proof of}
\begin{align}\label{ineq:proof-a-post-step-2}
	 \mu_0^{1/2}\|\varepsilon(\psi)\|_\mathbb{C} + \|\div(\sigma - \sigma_h)\|_{-1} \lesssim \eta.
\end{align}
Given any $\varphi \in V$,
let $\varphi_\mathrm{C} \in S^1_\mathrm{D}(\mathcal{T})$ denote the Scott-Zhang quasi-interpolation \cite{ScottZhang1990} of $\varphi$ with the local approximation property
\begin{align}\label{ineq:scott-zhang}
	h_T\|\varphi - \varphi_\mathrm{C}\|_{L^2(T)} + \|\D(\varphi - \varphi_\mathrm{C})\|_{L^2(T)} \lesssim \|\D \varphi\|_{L^2(\Omega(T))}
\end{align}
in the element patch $\Omega(T)$.
Since $\varphi_\mathrm{C} \in S^1_\mathrm{D}(\Tcal)$, $(\sigma - \sigma_h, \varepsilon(\varphi_\mathrm{C}))_{L^2(\Omega)} = 0$ from \eqref{ineq:data-osc} and \Cref{cor:kernel-stabilization}.
An integration by parts reveals
\begin{align*}
	&(\sigma - \sigma_h, \varepsilon (\varphi))_{L^2(\Omega)} = (\sigma - \sigma_h, \varepsilon(\varphi - \varphi_\mathrm{C}))_{L^2(\Omega)} = (\sigma - \sigma_h, \D(\varphi - \varphi_\mathrm{C}))_{L^2(\Omega)}\\
	&\quad = (f + \div_\pw \sigma_h, \varphi - \varphi_\mathrm{C})_{L^2(\Omega)}\\
	&\qquad\qquad + \sum_{F \in \mathcal{F}(\Omega)} (\varphi - \varphi_\mathrm{C}, [\sigma_h]_F \nu_F)_{L^2(F)} + \sum_{F \in \Fcal_\mathrm{N}} (\varphi - \varphi_\mathrm{C}, g - \sigma_h \nu_F)_{L^2(F)}.
\end{align*}
Standard arguments in the a~posteriori error analysis with the Cauchy and trace inequalities as well as with \eqref{ineq:scott-zhang} result in
\begin{align}\label{ineq:proof-a-post-step-1}
	(\sigma &- \sigma_h, \varepsilon (\varphi))_{L^2(\Omega)} \lesssim \|h_\mathcal{T}(f + \div_\pw \sigma_h)\|\|\D\varphi\|\nonumber\\
	& + \big(\sum_{F \in \mathcal{F}(\Omega)} h_F\|[\sigma_h]_F \nu_F\|_{L^2(F)}^2 + \sum_{F \in \Fcal_\mathrm{N}} h_F\|g - \sigma_h \nu_F\|_{L^2(F)}^2\big)^{1/2}\|\D \varphi\|.
\end{align}
This proves $\|\div(\sigma - \sigma_h)\|_{-1} \lesssim \eta$.
The definition of $\psi$ in \Cref{lem:tr-div-dev}.a implies $\|\varepsilon(\psi)\|_{\mathbb{C}}^2 = (\sigma - \sigma_h, \varepsilon(\psi))_{L^2(\Omega)}$. 
Since $\|\D \psi\| \lesssim \|\varepsilon(\psi)\| \lesssim \mu_0^{-1/2}\|\varepsilon(\psi)\|_\mathbb{C}$ from the (first) Korn inequality \eqref{ineq:Korn-inequality}, 
$\|\varepsilon(\psi)\|_\mathbb{C} \lesssim \mu_0^{-1/2}\eta$ follows from \eqref{ineq:proof-a-post-step-1} with the choice $\varphi \coloneqq \psi$.\\[-0.5em]

\noindent4.2.3. \emph{Proof of}
\begin{align}\label{ineq:proof-a-post-step-3}
	\mu_0^{1/2}\|\tau\|_{\mathbb{C}^{-1}} \lesssim (\mu_1/\mu_0)^{1/2}\eta.
\end{align}
The $L^2$ orthogonality $\mathbb{C}^{-1} \tau \perp \mathbb{C} \varepsilon (V)$ from the definition of $\psi$ in \Cref{lem:tr-div-dev}.d and a Cauchy inequality provide, for any $v \in V$, that
\begin{align}\label{ineq:proof-reliability-step-2}
	\|\tau\|_{\mathbb{C}^{-1}}^2 = (\mathbb{C} \varepsilon (v) - \sigma_h, \mathbb{C}^{-1} \tau)_{L^2(\Omega)} &= (\varepsilon (v) - \varepsilon_h u_h, \tau)_{L^2(\Omega)} \nonumber\\
	&\leq \mu_0^{-1/2}\|\mu^{1/2}(\varepsilon (v) - \varepsilon_h u_h)\| \|\tau\|.
\end{align}
Since $\sigma - \sigma_h$ and $\mathbb{C} \varepsilon(\psi)$ are contained in $\Sigma_0$ from \Cref{lem:tr-div-dev}.c--d, $\tau = \sigma - \sigma_h - \mathbb{C} \varepsilon(\psi) \in \Sigma_0$. Hence, \Cref{lem:tr-div-dev}.b implies $\|\tau\| \lesssim \|\mu^{1/2}\tau\|_{\mathbb{C}^{-1}} + \|\div\, \tau\|_{-1}$.
This, \eqref{ineq:proof-reliability-step-2}, and $\div\, \tau = 0$ conclude the proof of \eqref{ineq:proof-a-post-step-3}.\\[-0.5em]

\noindent\emph{Finish of the proof}
The bound $\|\sigma - \sigma_h\| \lesssim \|\mu^{1/2}(\sigma - \sigma_h)\|_{\mathbb{C}^{-1}} + \|\div(\sigma - \sigma_h)\|_{-1}$ from \eqref{ineq:L2-dev-div}, \eqref{eq:proof-reliability-split}--\eqref{ineq:proof-a-post-step-2}, and \eqref{ineq:proof-a-post-step-3} conclude the proof of $\|\sigma - \sigma_h\| \lesssim (\mu_1/\mu_0) \eta$.\hfill \qed
	
\subsection{Proof of efficiency}
We establish the efficiency of $\eta$ from \eqref{def:eta} with the choice $v \coloneqq \mathcal{A} u_h$.
The proof is divided into three parts: the first one is devoted to the efficiency of nodal averaging (without stabilization) while the remaining parts establish the efficiency of the remaining terms with bubble function techniques \cite{Verfuerth2013}.

\subsubsection{Efficiency of nodal averaging}
A straightforward modification of the proof of \cite[Theorem 4.7]{DiPietroDroniou2020} using \eqref{ineq:proof-conf-companion-3} and \Cref{lem:equiv-stab} leads to
\begin{align}\label{ineq:proof-efficiency-averaging}
	\begin{split}
		\|\D_\pw(\mathcal{R} - \mathcal{A}) u_h\| &\lesssim \min_{v \in V} \|\D_\pw (v - \mathcal{R} u_h)\| + \widetilde{\s}(u_h,u_h)^{1/2}\\
		&\lesssim \|\D_\pw(\mathcal{J} u_h - \mathcal{R} u_h)\| + |u_h|_{\widehat{\s}}
	\end{split}
\end{align}
with the right-inverse $\mathcal{J}$ of $\I$ from \Cref{lem:conforming-companion} and the stabilization $\widetilde{\s}$ from \eqref{def:alternative-stabilization}.
Since $\I \mathcal{J} u_h = u_h$, \eqref{ineq:ba-RI} provides
\begin{align*}
	\|\D_\pw(\mathcal{J} u_h - \mathcal{R} u_h)\| \lesssim \|\varepsilon_\pw(\mathcal{J} u_h - \mathcal{R} u_h)\|.
\end{align*}
Recall the best approximation property in \eqref{ineq:ba-RI} to obtain
\begin{align*}
	\|\varepsilon_\pw(\mathcal{J} u_h - \mathcal{R} u_h)\| \leq \|\varepsilon_\pw (\mathcal{J} u_h - \mathcal{R} \I u)\| \leq \|\varepsilon (\mathcal{J} u_h - \mathcal{J} \I u)\| + \|\varepsilon_\pw(\mathcal{J} - \mathcal{R}) \I u\|.
\end{align*}
The combination of the two previously displayed formula with
the stability $\|\varepsilon (\mathcal{J} u_h - \mathcal{J} \I u)\| \lesssim \|\I u - u_h\|_h$ and $\|\varepsilon_\pw(\mathcal{J} \I u - \mathcal{R} \I u)\| \lesssim \|\D_\pw(u - \mathcal{R} \I u)\| + |\I u|_{\widehat{\s}}$ from \Cref{lem:conforming-companion} results in
\begin{align*}
	\|\D_\pw(\mathcal{J} u_h - \mathcal{R} u_h)\| \lesssim \|\I u - u_h\|_h + \|\D_\pw(u - \mathcal{R} \I u)\| + |\I u|_{\widehat{\s}}.
\end{align*}
This, \eqref{ineq:proof-efficiency-averaging}, $\|\D_\pw(u - \mathcal{R} \I u)\| \lesssim \|\varepsilon_\pw(u - \mathcal{R} \I u)\|$ from \eqref{ineq:ba-RI}, $\|\I u - u_h\|_h \lesssim \mu_0^{-1/2}\|\I u - u_h\|_{a_h}$ from \Cref{thm:norm-equivalence}, and $\|\varepsilon_\pw(\mathcal{R} u_h - \mathcal{A} u_h)\| \leq \|\D_\pw(\mathcal{R} u_h - \mathcal{A} u_h)\|$ imply
\begin{align}\label{proof:eff-nodal-averaging}
	\|\varepsilon_\pw(\mathcal{R} u_h - \mathcal{A} u_h)\| \lesssim \|\varepsilon_\pw(u - \mathcal{R} \I u)\| + \mu_0^{-1/2}\|\I u - u_h\|_{a_h} + |u_h|_{\widehat{\s}} + |\I u|_{\widehat{\s}}.
\end{align}
From \Cref{lem:epsR-epsh} and a triangle inequality, we deduce that $\|\varepsilon_h u_h - \varepsilon(\mathcal{A} u_h)\| \lesssim \|\varepsilon_\pw(\mathcal{R} u_h - \mathcal{A} u_h)\| + |u_h|_{\widehat{\s}}$.
Hence, the a~priori result
$\mu_0^{1/2}|u_h|_\s \leq \mu_0^{1/2}\|\I u - u_h\|_{a_h} \lesssim \|\sigma - \sigma_h\| + \mathrm{osc}(f,\Tcal) + \mathrm{osc}(g,\Fcal_\mathrm{N})$ from \eqref{ineq:a-priori}, the quasi-optimality $|\I u|_{\widehat{\s}} + \|\varepsilon_\pw(u - \mathcal{R} \I u)\| \lesssim \|(1 - \Pi_{\Tcal}^k) \varepsilon(u)\|$ from \eqref{ineq:stability}, and \eqref{proof:eff-nodal-averaging} show
\begin{align}\label{ineq:eff-nodal-average}
	(\mu_0/\mu_1)\|\mu(\varepsilon_h u_h - \varepsilon(\mathcal{A} u_h))\| \lesssim \|\sigma - \sigma_h\| + \mathrm{osc}(f,\Tcal) + \mathrm{osc}(g,\Fcal_\mathrm{N}).
\end{align}

\subsubsection{Efficiency of Neumann data approximation}
Given $F \in \Fcal_\mathrm{N}$, let $T \in \Tcal$ be the unique simplex with $F \in \Fcal(T)$. Abbreviate $e_k \coloneqq \Pi_F^k g - \sigma_h \nu_F \in P_k(F)^n$. There exists $\varphi \in H^1(\Omega)^n$ with $\varphi = 0$ in $\Omega \setminus T$, $\Pi_F^k \varphi = h_F e_k$, and
\begin{align}\label{ineq:bubble}
	\|\D \varphi\|_{L^2(T)} \lesssim h_F^{-1}\|\varphi\|_{L^2(T)} \lesssim h_F^{-1/2}\|\varphi\|_{L^2(F)} \lesssim h_F^{1/2}\|e_k\|_{L^2(F)}
\end{align}
as in \cite[Subsection 4.3]{ErnZanotti2020}.
An integration by parts and a Cauchy inequality show
\begin{align*}
	((\sigma - \sigma_h)\nu_F, \varphi)_{L^2(F)}
	&= (f + \div\, \sigma_h, \varphi)_{L^2(T)} + (\sigma - \sigma_h, \D \varphi)_{L^2(T)}\\
	&\leq \|f + \div\, \sigma_h\|_{L^2(T)}\|\varphi\|_{L^2(T)} + \|\sigma - \sigma_h\|_{L^2(T)}\|\D \varphi\|_{L^2(T)}.
\end{align*}
Since $h_F\|e_k\|_{L^2(F)}^2 = ((\sigma - \sigma_h)\nu_F, \varphi)_{L^2(F)} + ((1 - \Pi_F^k) g, \varphi)_{L^2(F)}$ from $\sigma \nu_F = g$ on $F$, this, a Cauchy inequality, and \eqref{ineq:bubble} reveal
\begin{align*}
	h_F^{1/2}\|e_k\|_{L^2(F)} \lesssim h_T\|f + \div\, \sigma_h\|_{L^2(T)} + \|\sigma - \sigma_h\|_{L^2(T)} + h_F^{1/2}\|(1 - \Pi_F^k) g\|_{L^2(F)}.
\end{align*}
Therefore, the Pythagoras theorem provides
\begin{align}\label{ineq:efficiency-neumann-data}
	h_F\|g &- \sigma_h \nu_F\|^2_{L^2(F)} = h_F\|(1 - \Pi_F^k) g\|^2_{L^2(F)} + h_F\|e_k\|_{L^2(F)}^2\nonumber\\
	&\lesssim h_F\|(1 - \Pi_F^k) g\|^2_{L^2(F)} + \|h_\Tcal(f + \div \sigma_h)\|_{L^2(T)}^2 + \|\sigma - \sigma_h\|_{L^2(T)}^2.
\end{align}

\subsubsection{Finish of the proof}
The efficiency $h_T\|f + \div \sigma_h\|_{L^2(F)} \lesssim \|\sigma - \sigma_h\| + h_T\|(1 - \Pi_T^k) f\|_{L^2(T)}$ for any $T \in \Tcal$ has been established in \cite[Proof of Theorem 2]{BertrandCarstensenGraessle2021} for the Poisson equation. An extension to the case at hand is straightforward. The efficiency $h_F^{1/2}\|[\sigma_h] \nu_F\|_{L^2(F)}^2 \lesssim \|\sigma - \sigma_h\|_{L^2(\omega(F))} + \|h_\Tcal(f + \div_\pw \sigma_h)\|_{L^2(\omega(F))}$ can be established with the arguments from \cite[Section 1.4.5]{Verfuerth2013}.
Further details on these two terms are therefore omitted.
In combination with \eqref{ineq:eff-nodal-average} and \eqref{ineq:efficiency-neumann-data}, we conclude $(\mu_0/\mu_1)\eta \lesssim \|(1 - \Pi_{\Tcal}^k) \sigma\| + \mathrm{osc}(f,\Tcal) + \mathrm{osc}(g,\Fcal_\mathrm{N})$.\hfill\qed

\subsection{Extension to inhomogeneous Dirichlet boundary data in 2D}
In the following, we derive reliable error estimates for inhomogeneous Dirichlet boundary data in 2D. We assume that $\Gamma_{\mathrm{D}}$ lies on one connectivity component of $\partial \Omega$.
Given $u_\mathrm{D} \in C(\Gamma_{\mathrm{D}})^2 \cap H^1(\Fcal_\mathrm{D})^2$, $f \in L^2(\Omega)^2$, and $g \in L^2(\Gamma_{\mathrm{N}})^2$, the continuous (resp.~discrete) problem seeks the exact (resp.~discrete) solution $u \in u_\D + V$ (resp.~$u_h \in \I u_\D + V_h$) to \eqref{def:continuous-problem} (resp.~\eqref{def:discrete-problem}). 
(Here, $u_\mathrm{D}$ is understood as an $H^1$ extension in the domain $\Omega$.)
The proof of reliability in \Cref{sec:proof-reliability} and the observation $\mathbb{C}^{-1} \tau \perp \mathbb{C} \varepsilon(u_\mathrm{D} + V)$ (instead of $\mathbb{C}^{-1} \tau \perp \mathbb{C} \varepsilon(V)$ in \eqref{ineq:proof-reliability-step-2}) lead to
\begin{align}\label{ineq:a-post-inh1}
	\|\sigma - \sigma_h\| \lesssim (\mu_1/\mu_0)\eta
\end{align}
with $\eta$ from \eqref{def:eta}, where the minimum in \eqref{def:eta} is taken over the set $u_\mathrm{D} + V$ (instead of $V$). Let $\I_\mathrm{D} u_\mathrm{D}$ denote the nodal interpolation of $u_\mathrm{D}$ in $P_{k+1}(\Fcal_\mathrm{D})^2 \cap C(\Gamma_{\mathrm{D}})^2$. 
We choose the nodal average $v \coloneqq \mathcal{A} u_h$ of $\mathcal{R} u_h$ in \eqref{ineq:a-post-inh1}; the degrees of freedom of $\mathcal{A} u_h$ on $\Gamma_{\mathrm{D}}$ are fixed by the point evaluation of $u_\mathrm{D}$ so that $\mathcal{A} u_h|_{\Gamma_{\mathrm{D}}} = \I_\mathrm{D} u_\mathrm{D}$.
Since $\I_\mathrm{D} u_\mathrm{D} \neq u_\mathrm{D}$ in general, additional error quantities occur.
A triangle inequality shows
\begin{align}\label{ineq:a-post-inh2}
	\min_{w \in u_\mathrm{D} + V} \|\varepsilon(w) - \varepsilon_h u_h\| \leq \min_{w \in u_\mathrm{D} + V} \|\varepsilon(w) - \varepsilon(\mathcal{A} u_h)\| + \|\varepsilon(\mathcal{A} u_h) - \varepsilon_h u_h\|.
\end{align}
The minimizer $z$ of $\|\varepsilon(w) - \varepsilon(\mathcal{A} u_h)\|$ among $w \in u_\mathrm{D} + V$ defines $\varrho \coloneqq \varepsilon(z) - \varepsilon(\mathcal{A} u_h) \in L^2(\Omega)^2$ with the boundary condition $(z - \mathcal{A} u_h)|_{\Gamma_{\mathrm{D}}} = u_\mathrm{D} - \I_\mathrm{D} u_\mathrm{D}$ and $L^2$ orthogonality $\varrho \perp \D\,V$ from the Euler-Lagrange equations. Therefore, $\div\,\varrho = 0$ in $\Omega$ and $\varrho \nu = 0$ along $\Gamma_{\mathrm{N}}$. 
The Helmholtz decomposition \cite[Lemma 3.1]{CarstensenDolzmann1998} shows the existence of $\beta = (\beta_1,\beta_2) \in H^1(\Omega)^2$ (even as a $\mathrm{Curl}$ of an $H^2$ function) with
\begin{align}\label{ineq:a-post-inh3}
	\mathrm{Curl}\,\beta = \begin{pmatrix}
		\partial_2 \beta_1 & - \partial_1 \beta_1\\
		\partial_2 \beta_2 & - \partial_1 \beta_2\\
	\end{pmatrix} = \varrho, \qquad \|\beta\|_{H^1(\Omega)} \lesssim \|\varrho\|, \qquad\text{and}
\end{align}
$\beta$ is constant on each connectivity component of $\Gamma_{\mathrm{N}}$, i.e., $\partial \beta/\partial s = 0$ on $\mathrm{relint}(\Gamma_{\mathrm{N}})$.
(The assumptions in \cite{CarstensenDolzmann1998} that (a) $\Gamma_{\mathrm{D}}$ is connected and (b) $\Gamma_{\mathrm{N}}$ and $\Gamma_{\mathrm{D}}$ have positive distance, can be relaxed to the present one with the arguments from here and \cite{CarstensenDolzmann1998}.)
The a.e.~pointwise symmetry of $\mathrm{Curl}\,\beta$ and an integration by parts reveal
\begin{align*}
	\|\varrho\|^2 &= (\varrho, \mathrm{Curl}\,\beta)_{L^2(\Omega)} = (\D z - \D \mathcal{A} u_h, \mathrm{Curl}\,\beta)_{L^2(\Omega)}\\
	&= -\int_{\partial \Omega} (z - \mathcal{A} u_h) \cdot \partial \beta/\partial s \d{s} = -\int_{\Gamma_\mathrm{D}} (u_\mathrm{D} - \I_\mathrm{D} u_\mathrm{D}) \cdot \partial \beta/\partial s \d{s}\\
	&= -\sum_{F \in \Fcal_\mathrm{D}} \int_F (u_\mathrm{D} - \I_\mathrm{D} u_\mathrm{D}) \cdot \partial \beta/\partial s \d{s},
\end{align*}
because $(u_\mathrm{D} - \I_\mathrm{D} u_\mathrm{D})|_F \in H^1_0(F)^2$ on $F \in \Fcal_\mathrm{D}$ and $\partial \beta/\partial s \in H^{-1/2}(\partial \Omega)$ vanishes on $\Gamma_{\mathrm{N}}$ (written $\partial \beta/\partial s \in \widetilde{H}^{-1/2}(\Gamma_{\mathrm{D}})^2$ as its zero extension belongs to $H^{-1/2}(\partial \Omega)$). This holds at least for smooth $\beta$, where an edge-wise integration by parts shows
\begin{align}\label{ineq:rho-integration-by-parts}
	\|\varrho\|^2 = (\D(z - \mathcal{A} u_h), \mathrm{Curl}\,\beta)_{L^2(\Omega)} = \sum_{F \in \Fcal_\mathrm{D}} \int_{F} \beta \cdot \partial (u_\mathrm{D} - \I_\mathrm{D} u_\mathrm{D})/\partial s \d{s}.
\end{align}
The second equation (4.15.b) in \eqref{ineq:rho-integration-by-parts} holds for all smooth functions $\beta$ with $\partial \beta/\partial s = 0$ on $\Gamma_{\mathrm{N}}$. Since those functions can approximate $\beta \in H^1(\Omega)^2$ from \eqref{ineq:a-post-inh3} (by density in $H^1(\Omega)^2$ even with the side restriction that $\beta$ is constant on each connectivity component of $\Gamma_{\mathrm{N}}$), (4.15.b) holds for $\beta$ from \eqref{ineq:a-post-inh3} as well.
Nodal interpolation leads to $\int_F  \partial (u_\mathrm{D} - \I_\mathrm{D} u_\mathrm{D})/\partial s \d{s} = 0$ 
for any $F \in \Fcal_\mathrm{D}$ with the fundamental theorem of calculus along the 1D edge $F$.
This and the Cauchy inequality in \eqref{ineq:rho-integration-by-parts} provide
\begin{align}\label{ineq:rho-bound}
	\|\varrho\|^2 &\leq \sum_{F \in \Fcal_\mathrm{D}} \|(1 - \Pi_F^0) \beta\|_{L^2(F)} \|\partial(u_\mathrm{D} - \I_\mathrm{D} u_\mathrm{D})/\partial \tau\|_{L^2(F)}\nonumber\\
	&\leq \big(\sum_{F \in \Fcal_\mathrm{D}} h_F^{-1}\|(1 - \Pi_F^0) \beta\|_{L^2(F)}^2\big)^{1/2} \mathrm{osc}(u_\mathrm{D},\Fcal_\mathrm{D})
\end{align}
with $\mathrm{osc}(u_\mathrm{D},\Fcal_\mathrm{D})^2 \coloneqq \sum_{F \in \Fcal_\mathrm{D}} h_F\|\partial(u_\mathrm{D} - \I_\mathrm{D} u_\mathrm{D})/\partial s\|_{L^2(F)}^2$.
Given $F \in \Fcal_\mathrm{D}$, let $T_F \in \Tcal$ be the unique triangle with $F \in \Fcal(T_F)$. A trace and Poincar\'e inequality imply
$$
	\sum_{F \in \Fcal_\mathrm{D}} h_F^{-1}\|(1 - \Pi_F^0) \beta\|_{L^2(F)}^2 \leq \sum_{F \in \Fcal_\mathrm{D}} h_F^{-1}\|(1 - \Pi_{T_F}^0) \beta\|_{L^2(F)}^2 \lesssim \|\D \beta\|^2 \lesssim \|\varrho\|^2
$$
with \eqref{ineq:a-post-inh3} in the last step.
From this and \eqref{ineq:rho-bound}, we infer
\begin{align*}
	\|\varrho\| \lesssim \mathrm{osc}(u_\mathrm{D},\Fcal_\mathrm{D}).
\end{align*}
This and \eqref{ineq:a-post-inh1}--\eqref{ineq:a-post-inh2} reveal the reliability of the error estimator $\widetilde{\eta}$,
\begin{align}\label{ineq:eta-inh}
	(\mu_0/\mu_1)^2\|\sigma &- \sigma_h\|^2 \lesssim \widetilde{\eta}^2 \coloneqq \|h_\Tcal(f + \div_\pw \sigma_h)\|^2 + \|\mu(\varepsilon(\mathcal{A} u_h) - \varepsilon_h u_h)\|^2\nonumber\\
	 + \sum_{F \in \Fcal(\Omega)} &h_F\|[\sigma_h]_F \nu_F\|_{L^2(F)}^2 +
	\sum_{F \in \Fcal_\mathrm{N}} h_F\|g - \sigma_h \nu_F\|_{L^2(F)}^2 + \mu_1^2\mathrm{osc}(u_\mathrm{D},\Fcal_\mathrm{D})^2.
\end{align}

\section{Numerical examples}\label{sec:numerical-examples}
This section provides two numerical benchmarks in 2D with locking behaviour observed in \cite{CarstensenEigelGedicke2011} for low-order conforming FEM.

\subsection{Preliminary remarks}
The material parameters $\lambda, \mu$ are given by the formulas $\lambda = E\nu/((1+\nu)(1-2\nu))$ and $\mu = E/(2(1+\nu)) \to \infty$ as $\nu \to 1/2$ with the default Young's modulus $E = 10^5$ and the Poisson ratio $\nu = 0.4999$. 
The only exceptions of those parameters are in \Cref{fig:Cooks-membrane-2}.b and \Cref{fig:Rotated-Lshape-3}.b, where $\nu=0.4999$ is displayed as well as numerical results for $\nu=0.3$ for comparison.

Adaptive computations utilize the refinement indicator
\begin{align}\label{def:local-eta}
	\eta^2(T) &\coloneqq \|h_T(f + \div\, \sigma_h)\|_{L^2(T)}^2 + (\mu|_T)^2 \|\varepsilon (\mathcal{A} u_h) - \varepsilon_h u_h\|_{L^2(T)}^2\nonumber\\
	&\qquad + \sum_{F \in \Fcal(T) \cap \Fcal(\Omega)} h_F\|[\sigma_h]_F \nu_F\|_{L^2(F)}^2 +
	\sum_{F \in \Fcal(T) \cap \Fcal_\mathrm{N}} h_F\|g - \sigma_h \nu_F\|_{L^2(F)}^2
\end{align}
for $T \in \Tcal$ with $\Tcal \in \mathbb{T}$
arising from the error estimator \eqref{def:eta} with the average $\mathcal{A} u_h$ of $\mathcal{R} u_h$.
A standard adaptive loop \cite{Dorfler1996} with the D\"orfler
marking strategy determines a subset $\mathcal{M}$ of minimal cardinality such that
\begin{align*}
	\sum_{T \in \mathcal{T}} \eta^2(T) \leq \frac{1}{2} \sum_{T \in \mathcal{M}} \eta^2(T).
\end{align*}
The convergence history plots display the quantities of interest against the number of degrees of freedom ($\mathtt{ndof}$) in a log-log plot. (Recall the scaling $\mathtt{ndof} \approx h_\mathrm{max}^{-2}$ for uniform meshes in 2D.) Solid lines indicate adaptive, while dashed lines are associated with uniform mesh refinements.

The implementation of the solver for the discrete problem \eqref{def:discrete-problem} has been carried out in MATLAB in an extension to the short MATLAB programs \cite{AlbertyCFunken1999}.
The integrals of polynomials are exactly computed while numerical quadrature is utilized for the integration of non-polynomial functions, e.g., in the computation of the errors.

\begin{figure}[ht]
	\begin{minipage}[t]{0.4\textwidth}
		\centering
		\includegraphics[height=5cm]{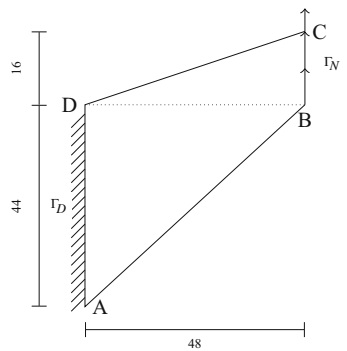}
	\end{minipage}\hfill
	\begin{minipage}[t]{0.6\textwidth}
		\centering
		\includegraphics[height=5cm]{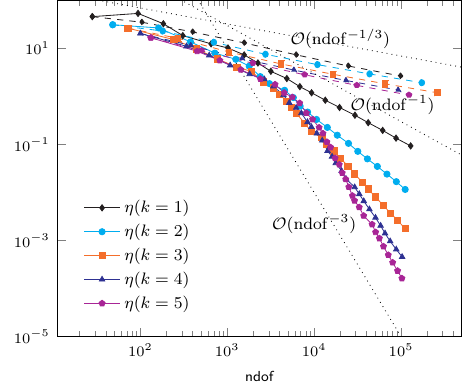}
	\end{minipage}
	\captionsetup{width=1\linewidth}
	\caption{(a) Cook's membrane and (b) convergence history plot of $\eta$ for $k = 1, \dots, 5$ in \Cref{sec:Cooks}}
	\label{fig:Cooks-membrane-1}
\end{figure}

\begin{figure}[ht]
	\begin{minipage}[t]{0.475\textwidth}
		\centering
		\includegraphics[height=5cm]{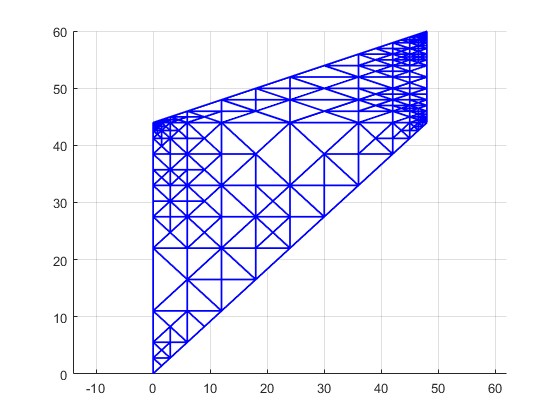}
	\end{minipage}\hfill
	\begin{minipage}[t]{0.475\textwidth}
		\centering
		\includegraphics[height=5cm]{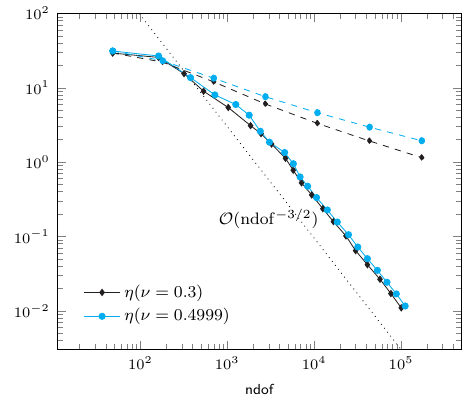}
	\end{minipage}
	\captionsetup{width=1\linewidth}
	\caption{(a) Adaptive triangulation into 592 triangles generated with $k = 2$ and (b) comparison between different $\nu$ with $k = 2$ in \Cref{sec:Cooks}}
	\label{fig:Cooks-membrane-2}
\end{figure}

\subsection{Cook's membrane}\label{sec:Cooks}
The tapered panel $\Omega \coloneqq \mathrm{conv}\{(0,0), (48,44), (48,60)$, $(0,44)\}$ of \Cref{fig:Cooks-membrane-1}.a (courtesy from \cite{CarstensenGallistlGedicke2019}) is clamped on the left-hand side $\Gamma_{\mathrm{D}} \coloneqq \mathrm{conv}\{(0,0),(0,44)\}$, subjected to the shear load $g = (0,1)$ in vertical direction on the right side $\{48\} \times [44,60]$, and traction free elsewhere.
The initial triangulation $\mathcal{T}_0$ consists of two triangles by partition of $\Omega$ along the line $[0,48] \times \{44\}$.
\Cref{fig:Cooks-membrane-1}.b displays the suboptimal convergence rate $1/3$ for the error estimator $\eta$ with all displayed polynomial degrees $k = 1, \dots, 5$. The adaptive algorithm locally refines towards the four corners of the domain $\Omega$ in \Cref{fig:Cooks-membrane-2}.a for $k = 2$ and $|\Tcal| = 592$ and recovers the optimal convergence rates $(k+1)/2$ for $\eta$ in \Cref{fig:Cooks-membrane-1}.b.

\begin{figure}[ht]
	\begin{minipage}[t]{0.475\textwidth}
		\centering
		\includegraphics[height=5cm]{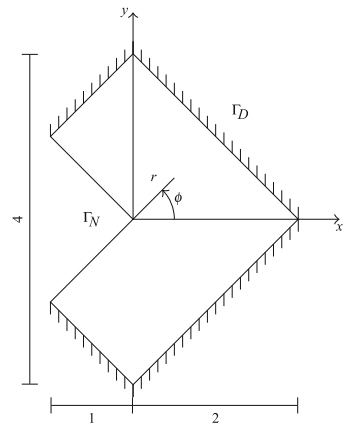}
	\end{minipage}\hfill
	\begin{minipage}[t]{0.475\textwidth}
		\centering
		\includegraphics[height=5cm]{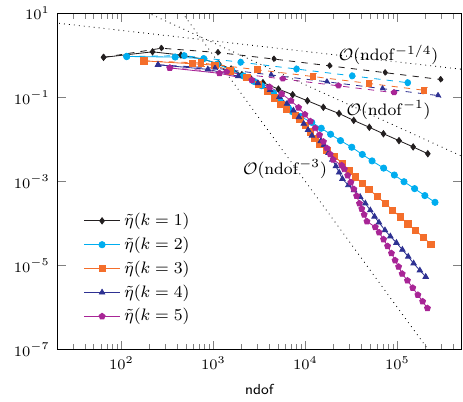}
	\end{minipage}
	\captionsetup{width=1\linewidth}
	\caption{(a) Rotated L-shaped domain and (b) convergence history plot of $\widetilde{\eta}$ for $k = 1,\dots,5$ in \Cref{sec:Lshape}}
	\label{fig:Rotated-Lshape-1}
\end{figure}
\begin{figure}[ht]
	\begin{minipage}[t]{0.475\textwidth}
		\centering
		\includegraphics[height=5cm]{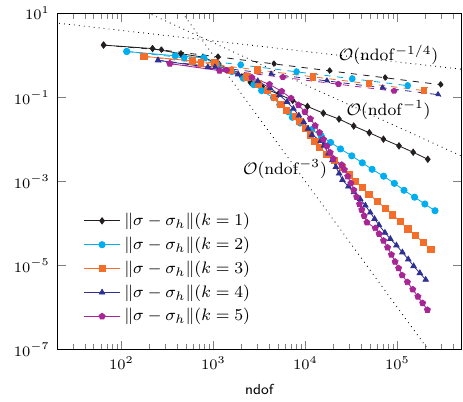}
	\end{minipage}\hfill
	\begin{minipage}[t]{0.475\textwidth}
		\centering
		\includegraphics[height=5cm]{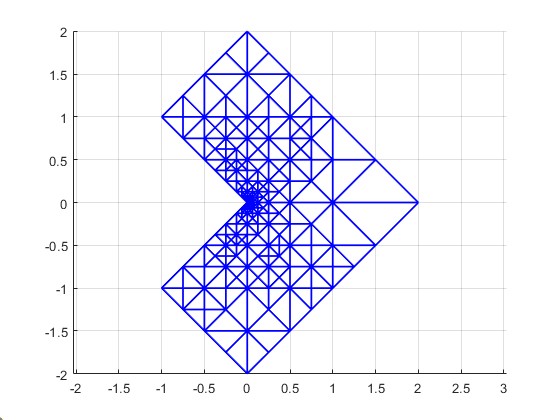}
	\end{minipage}\hfill
	\captionsetup{width=1\linewidth}
	\caption{(a) Convergence history plot of $\|\sigma - \sigma_h\|$ for $k = 1, \dots, 5$ and (b) adaptive triangulation into 545 triangles generated with $k = 2$ in \Cref{sec:Lshape}}
	\label{fig:Rotated-Lshape-2}
\end{figure}
\begin{figure}[ht]
	\begin{minipage}[t]{0.475\textwidth}
		\centering
		\includegraphics[height=5cm]{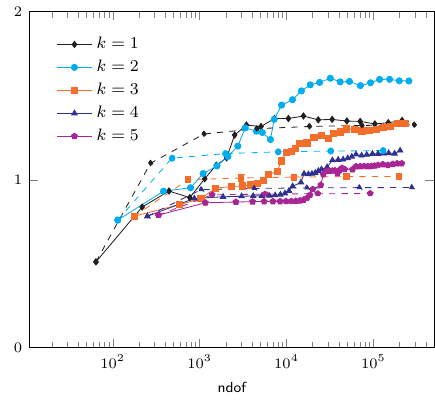}
	\end{minipage}\hfill
	\begin{minipage}[t]{0.475\textwidth}
		\centering
		\includegraphics[height=5cm]{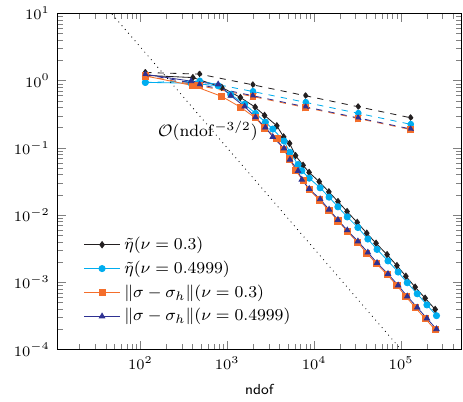}
	\end{minipage}
	\captionsetup{width=1\linewidth}
	\caption{(a) Efficiency index $\widetilde{\eta}/\|\sigma - \sigma_h\|$ for $k = 1,\dots,5$ and (b) comparison between different $\nu$ with $k = 2$ in \Cref{sec:Lshape}}
	\label{fig:Rotated-Lshape-3}
\end{figure}

\subsection{L-shaped domain}\label{sec:Lshape}
The rotated L-shaped domain $\Omega = \mathrm{conv}\{(0,0),(-1,-1)$, $(0,-2),(1,-1),(2,0),(0,2),(-1,1)\}$
with displayed Dirichlet and Neumann boundary parts
in \Cref{fig:Rotated-Lshape-1}.a (courtesy from \cite{CarstensenGallistlGedicke2019}) allows for
the exact solution $u$ to \eqref{def:continuous-problem}, given in polar coordinates,
\begin{align*}
	u_r(r,\varphi) &= \frac{r^\alpha}{2\mu}(-(\alpha+1)\cos((\alpha+1)\varphi) + (C_2-\alpha-1)C_1\cos((\alpha-1)\varphi)),\\
	u_\varphi(r,\varphi) &=
	\frac{r^\alpha}{2\mu}((\alpha+1)\sin((\alpha+1)\varphi) + (C_2+\alpha-1)C_1\sin((\alpha-1)\varphi))
\end{align*}
with the first root $\alpha = 0.544483736782$ of $\alpha \sin(2\omega) + \sin(2\omega\alpha) = 0$ for $\omega = 3\pi/4$,
$C_1 = -\cos((\alpha+1)\omega)/\cos((\alpha-1)\omega)$, and
$C_2 = 2(\lambda+2\mu)/(\lambda+\mu)$ \cite{CarstensenEigelGedicke2011}.
The applied force $f = 0$ and $g = 0$ vanish but inhomogeneous Dirichlet boundary conditions apply with $\|\sigma - \sigma_h\| \lesssim \widetilde{\eta}$ from \eqref{ineq:eta-inh}. (In particular, the local contributions of $\mathrm{osc}(u_\mathrm{D},\Fcal_\mathrm{D})^2$ are added to \eqref{def:local-eta}.)

\Cref{fig:Rotated-Lshape-1}.b and \Cref{fig:Rotated-Lshape-2}.a display convergence rates $1/4$ for $\widetilde{\eta}$ and $\|\sigma - \sigma_h\|$ on uniform triangulations. Adaptive computations refine towards the singularity at the origin as shown in \Cref{fig:Rotated-Lshape-2}.b and recover the optimal convergence rates $(k+1)/2$ for $\widetilde{\eta}$ and $\|\sigma - \sigma_h\|$ with all displayed polynomial degrees $k$.
In this example, the efficiency indices are in the range 0.5 to 2 and do not increase with larger $k$, cf.~\Cref{fig:Rotated-Lshape-3}.a.

\subsection{Conclusions}
In all three numerical examples, the a~posteriori error estimator $\eta$ is reliable, efficient, and $\lambda$-robust. This allows for the approximation of the Stokes equations in the incompressible limit as $\lambda \to \infty$.
The adaptive algorithm driven by the local contributions of $\eta$ recovers the optimal convergence rates for the approximation of singular
solutions and enables higher order convergence rates in typical computational benchmarks.
A comparison with the results for $\nu = 0.3$ displayed in \Cref{fig:Cooks-membrane-2}.b and \Cref{fig:Rotated-Lshape-3}.b verifies empirically the $\lambda$-robustness of the method.

\printbibliography

@article {DiPietroErnLemaire2014,
	AUTHOR = {Di Pietro, Daniele A. and Ern, Alexandre and Lemaire, Simon},
	TITLE = {An arbitrary-order and compact-stencil discretization of
	diffusion on general meshes based on local reconstruction
	operators},
	JOURNAL = {Comput. Methods Appl. Math.},
	FJOURNAL = {Computational Methods in Applied Mathematics},
	VOLUME = {14},
	YEAR = {2014},
	NUMBER = {4},
	PAGES = {461--472},
	ISSN = {1609-4840},
	MRCLASS = {65N08 (65N12)},
	MRREVIEWER = {Andrei I. Tolstykh},
	DOI = {10.1515/cmam-2014-0018},
	URL = {https://doi.org/10.1515/cmam-2014-0018},
}

@article {ErnZanotti2020,
	AUTHOR = {Ern, Alexandre and Zanotti, Pietro},
	TITLE = {A quasi-optimal variant of the hybrid high-order method for
	elliptic partial differential equations with {$H^{-1}$} loads},
	JOURNAL = {IMA J. Numer. Anal.},
	FJOURNAL = {IMA Journal of Numerical Analysis},
	VOLUME = {40},
	YEAR = {2020},
	NUMBER = {4},
	PAGES = {2163--2188},
	ISSN = {0272-4979},
	MRCLASS = {65N99 (65N15)},
	MRREVIEWER = {Samuel Cogar},
	DOI = {10.1093/imanum/drz057},
	URL = {https://doi.org/10.1093/imanum/drz057},
}

@article {DiPietroErn2015,
	AUTHOR = {Di Pietro, Daniele A. and Ern, Alexandre},
	TITLE = {A hybrid high-order locking-free method for linear elasticity
	on general meshes},
	JOURNAL = {Comput. Methods Appl. Mech. Engrg.},
	FJOURNAL = {Computer Methods in Applied Mechanics and Engineering},
	VOLUME = {283},
	YEAR = {2015},
	PAGES = {1--21},
	ISSN = {0045-7825},
	MRCLASS = {65N30 (74B05 74G15)},
	MRREVIEWER = {Alexandre L. Madureira},
	DOI = {10.1016/j.cma.2014.09.009},
	URL = {https://doi.org/10.1016/j.cma.2014.09.009},
}

@article {ScottZhang1990,
	AUTHOR = {Scott, L. Ridgway and Zhang, Shangyou},
	TITLE = {Finite element interpolation of nonsmooth functions satisfying
	boundary conditions},
	JOURNAL = {Math. Comp.},
	FJOURNAL = {Mathematics of Computation},
	VOLUME = {54},
	YEAR = {1990},
	NUMBER = {190},
	PAGES = {483--493},
	ISSN = {0025-5718},
	MRCLASS = {65D05 (65N30)},
	MRREVIEWER = {Qian Li},
	DOI = {10.2307/2008497},
	URL = {https://doi.org/10.2307/2008497},
}

@book {GiraultRaviart1986,
	AUTHOR = {Girault, Vivette and Raviart, Pierre-Arnaud},
	TITLE = {Finite element methods for {N}avier-{S}tokes equations},
	VOLUME = {5},
	NOTE = {Theory and algorithms},
	PUBLISHER = {Springer-Verlag, Berlin},
	YEAR = {1986},
	ISBN = {3-540-15796-4},
	MRCLASS = {65N30 (65-02 76-08)},
	MRREVIEWER = {Max D. Gunzburger},
	DOI = {10.1007/978-3-642-61623-5},
	URL = {https://doi.org/10.1007/978-3-642-61623-5},
}

@article {CarstensenZhaiZhang2020,
	AUTHOR = {Carstensen, Carsten and Zhai, Qilong and Zhang, Ran},
	TITLE = {A skeletal finite element method can compute lower eigenvalue
	bounds},
	JOURNAL = {SIAM J. Numer. Anal.},
	FJOURNAL = {SIAM Journal on Numerical Analysis},
	VOLUME = {58},
	YEAR = {2020},
	NUMBER = {1},
	PAGES = {109--124},
	ISSN = {0036-1429},
	MRCLASS = {65N25 (65N15 65N30)},
	MRREVIEWER = {Kathrin Smetana},
	DOI = {10.1137/18M1212276},
	URL = {https://doi.org/10.1137/18M1212276},
}

@article {BertrandCarstensenGraessle2021,
    AUTHOR = {Bertrand, Fleurianne and Carstensen, Carsten and Gr\"{a}{\ss}le, Benedikt and Tran, Ngoc Tien},
     TITLE = {Stabilization-free {HHO} a posteriori error control},
   JOURNAL = {Numer. Math.},
  FJOURNAL = {Numerische Mathematik},
    VOLUME = {154},
      YEAR = {2023},
    NUMBER = {3-4},
     PAGES = {369--408},
      ISSN = {0029-599X,0945-3245},
   MRCLASS = {65N12 (65N30 65Y20)},
  MRNUMBER = {4630545},
       DOI = {10.1007/s00211-023-01366-8},
       URL = {https://doi.org/10.1007/s00211-023-01366-8},
}

@book {Verfuerth2013,
	AUTHOR = {Verf\"{u}rth, R\"{u}diger},
	TITLE = {A posteriori error estimation techniques for finite element
	methods},
	PUBLISHER = {Oxford University Press, Oxford},
	YEAR = {2013},
	ISBN = {978-0-19-967942-3},
	MRCLASS = {65N30 (35J25 35K20 35Q30 35Q74 65N15)},
	MRREVIEWER = {Manfred Dobrowolski},
	DOI = {10.1093/acprof:oso/9780199679423.001.0001},
	URL = {https://doi.org/10.1093/acprof:oso/9780199679423.001.0001},
}

@article {AlbertyCFunken1999,
	AUTHOR = {Alberty, Jochen and Carstensen, Carsten and Funken, Stefan A.},
	TITLE = {Remarks around 50 lines of {M}atlab: short finite element
	implementation},
	JOURNAL = {Numer. Algorithms},
	FJOURNAL = {Numerical Algorithms},
	VOLUME = {20},
	YEAR = {1999},
	NUMBER = {2-3},
	PAGES = {117--137},
	ISSN = {1017-1398},
	MRCLASS = {65N30 (65M60)},
	DOI = {10.1023/A:1019155918070},
	URL = {https://doi.org/10.1023/A:1019155918070},
}

@book {BrennerScott2008,
	AUTHOR = {Brenner, Susanne C. and Scott, L. Ridgway},
	TITLE = {The mathematical theory of finite element methods},
	VOLUME = {15},
	EDITION = {Third},
	PUBLISHER = {Springer, New York},
	YEAR = {2008},
	ISBN = {978-0-387-75933-3},
	MRCLASS = {65-01 (65-02)},
	DOI = {10.1007/978-0-387-75934-0},
	URL = {https://doi.org/10.1007/978-0-387-75934-0},
}

@book {ErnGuermondII2021,
    AUTHOR = {Ern, Alexandre and Guermond, Jean-Luc},
     TITLE = {Finite elements {I}---{A}pproximation and interpolation},
    VOLUME = {72},
 PUBLISHER = {Springer},
      YEAR = {2021},
      ISBN = {978-3-030-56340-0; 978-3-030-56341-7},
   MRCLASS = {65-01},
       DOI = {10.1007/978-3-030-56341-7},
       URL = {https://doi.org/10.1007/978-3-030-56341-7},
}

@book {BoffiBrezziFortin2013,
    AUTHOR = {Boffi, Daniele and Brezzi, Franco and Fortin, Michel},
     TITLE = {Mixed finite element methods and applications},
    VOLUME = {44},
 PUBLISHER = {Springer, Heidelberg},
      YEAR = {2013},
      ISBN = {978-3-642-36518-8; 978-3-642-36519-5},
   MRCLASS = {65-02 (65M60 65N30)},
MRREVIEWER = {Beny Neta},
       DOI = {10.1007/978-3-642-36519-5},
       URL = {https://doi.org/10.1007/978-3-642-36519-5},
}

@article {VeeserZanottiIII,
    AUTHOR = {Veeser, Andreas and Zanotti, Pietro},
     TITLE = {Quasi-optimal nonconforming methods for symmetric elliptic
              problems. {III}---{D}iscontinuous {G}alerkin and other
              interior penalty methods},
   JOURNAL = {SIAM J. Numer. Anal.},
  FJOURNAL = {SIAM Journal on Numerical Analysis},
    VOLUME = {56},
      YEAR = {2018},
    NUMBER = {5},
     PAGES = {2871--2894},
      ISSN = {0036-1429},
   MRCLASS = {65N30 (65N12 65N15)},
  MRNUMBER = {3857891},
MRREVIEWER = {Francesco Zirilli},
       DOI = {10.1137/17M1151675},
       URL = {https://doi.org/10.1137/17M1151675},
}

@article {VeeserZanottiII,
    AUTHOR = {Veeser, Andreas and Zanotti, Pietro},
     TITLE = {Quasi-optimal nonconforming methods for symmetric elliptic
              problems. {II}---{O}verconsistency and classical nonconforming
              elements},
   JOURNAL = {SIAM J. Numer. Anal.},
  FJOURNAL = {SIAM Journal on Numerical Analysis},
    VOLUME = {57},
      YEAR = {2019},
    NUMBER = {1},
     PAGES = {266--292},
      ISSN = {0036-1429},
   MRCLASS = {65N30 (65N12 65N15)},
  MRNUMBER = {3907927},
MRREVIEWER = {Sergio Caucao},
       DOI = {10.1137/17M1151651},
       URL = {https://doi.org/10.1137/17M1151651},
}

@article {BottiDiPietroSochala2017,
    AUTHOR = {Botti, Michele and Di Pietro, Daniele A. and Sochala, Pierre},
     TITLE = {A hybrid high-order method for nonlinear elasticity},
   JOURNAL = {SIAM J. Numer. Anal.},
  FJOURNAL = {SIAM Journal on Numerical Analysis},
    VOLUME = {55},
      YEAR = {2017},
    NUMBER = {6},
     PAGES = {2687--2717},
      ISSN = {0036-1429,1095-7170},
   MRCLASS = {65N30 (65N08 74B20 74S05)},
  MRNUMBER = {3721565},
MRREVIEWER = {Corina-\c{S}tefania\ Drapaca},
       DOI = {10.1137/16M1105943},
       URL = {https://doi.org/10.1137/16M1105943},
}

@book {Braess2007,
    AUTHOR = {Braess, Dietrich},
     TITLE = {Finite elements},
   EDITION = {Third},
      NOTE = {Theory, fast solvers, and applications in elasticity theory,
              Translated from the German by Larry L. Schumaker},
 PUBLISHER = {Cambridge University Press, Cambridge},
      YEAR = {2007},
     PAGES = {xviii+365},
      ISBN = {978-0-521-70518-9; 0-521-70518-5},
   MRCLASS = {65N30 (65-02 74S05)},
  MRNUMBER = {2322235},
       DOI = {10.1017/CBO9780511618635},
       URL = {https://doi.org/10.1017/CBO9780511618635},
}

@article {CarstensenDolzmann1998,
    AUTHOR = {Carstensen, Carsten and Dolzmann, Georg},
     TITLE = {A posteriori error estimates for mixed {FEM} in elasticity},
   JOURNAL = {Numer. Math.},
  FJOURNAL = {Numerische Mathematik},
    VOLUME = {81},
      YEAR = {1998},
    NUMBER = {2},
     PAGES = {187--209},
      ISSN = {0029-599X,0945-3245},
   MRCLASS = {65N30 (65N15 73V05)},
  MRNUMBER = {1657768},
MRREVIEWER = {Petr\ Proch\'{a}zka},
       DOI = {10.1007/s002110050389},
       URL = {https://doi.org/10.1007/s002110050389},
}

@book {DiPietroDroniou2020,
    AUTHOR = {Di Pietro, Daniele Antonio and Droniou, J\'{e}r\^{o}me},
     TITLE = {The hybrid high-order method for polytopal meshes},
    SERIES = {MS\&A. Modeling, Simulation and Applications},
    VOLUME = {19},
      NOTE = {Design, analysis, and applications},
 PUBLISHER = {Springer},
      YEAR = {2020},
     PAGES = {xxxi+525},
      ISBN = {978-3-030-37202-6; 978-3-030-37203-3},
   MRCLASS = {65-02 (74-01 76-01 78Mxx)},
  MRNUMBER = {4230986},
       DOI = {10.1007/978-3-030-37203-3},
       URL = {https://doi.org/10.1007/978-3-030-37203-3},
}

@article {Brenner2004,
    AUTHOR = {Brenner, Susanne C.},
     TITLE = {Korn's inequalities for piecewise {$H^1$} vector fields},
   JOURNAL = {Math. Comp.},
  FJOURNAL = {Mathematics of Computation},
    VOLUME = {73},
      YEAR = {2004},
    NUMBER = {247},
     PAGES = {1067--1087},
      ISSN = {0025-5718,1088-6842},
   MRCLASS = {65N30 (26D15 35J55 74S05)},
  MRNUMBER = {2047078},
MRREVIEWER = {Thomas\ Apel},
       DOI = {10.1090/S0025-5718-03-01579-5},
       URL = {https://doi.org/10.1090/S0025-5718-03-01579-5},
}

@article {Maubach1995,
    AUTHOR = {Maubach, Joseph M.},
     TITLE = {Local bisection refinement for {$n$}-simplicial grids
              generated by reflection},
   JOURNAL = {SIAM J. Sci. Comput.},
  FJOURNAL = {SIAM Journal on Scientific Computing},
    VOLUME = {16},
      YEAR = {1995},
    NUMBER = {1},
     PAGES = {210--227},
      ISSN = {1064-8275},
   MRCLASS = {65M50},
  MRNUMBER = {1311687},
MRREVIEWER = {Patrick\ M.\ Knupp},
       DOI = {10.1137/0916014},
       URL = {https://doi.org/10.1137/0916014},
}

@Article{Stevenson2008,
	Author = {Stevenson, Rob},
	Title = {The completion of locally refined simplicial partitions created by bisection},
	FJournal = {Mathematics of Computation},
	Journal = {Math. Comput.},
	Volume = {77},
	Number = {261},
	Pages = {227--241},
	Year = {2008},
}

@article{Dorfler1996,
author = {D\"{o}rfler, Willy},
title = {A convergent adaptive algorithm for Poisson’s equation},
journal = {SIAM J. Numer. Anal.},
volume = {33},
number = {3},
pages = {1106-1124},
year = {1996},
}

@article {CarstensenGallistlGedicke2019,
    AUTHOR = {Carstensen, Carsten and Gallistl, Dietmar and Gedicke, Joscha},
     TITLE = {Residual-based a posteriori error analysis for symmetric mixed
              {A}rnold-{W}inther {FEM}},
   JOURNAL = {Numer. Math.},
  FJOURNAL = {Numerische Mathematik},
    VOLUME = {142},
      YEAR = {2019},
    NUMBER = {2},
     PAGES = {205--234},
      ISSN = {0029-599X,0945-3245},
   MRCLASS = {65N30 (65N15)},
  MRNUMBER = {3941930},
MRREVIEWER = {Xiaobing\ Henry\ Feng},
       DOI = {10.1007/s00211-019-01029-7},
       URL = {https://doi.org/10.1007/s00211-019-01029-7},
}

@book {CicuttinErnPignet2021,
    AUTHOR = {Cicuttin, Matteo and Ern, Alexandre and Pignet, Nicolas},
     TITLE = {Hybrid high-order methods---a primer with applications to
              solid mechanics},
    SERIES = {Springer Briefs in Mathematics},
 PUBLISHER = {Springer, Cham},
      YEAR = {2021},
     PAGES = {viii+136},
      ISBN = {978-3-030-81476-2; 978-3-030-81477-9},
   MRCLASS = {65N30 (65N12 65N15 74S05)},
  MRNUMBER = {4387067},
       DOI = {10.1007/978-3-030-81477-9},
       URL = {https://doi.org/10.1007/978-3-030-81477-9},
}

@article {CN21a,
    AUTHOR = {Carstensen, C. and Nataraj, N.},
     TITLE = {A priori and a posteriori error analysis of the
              {C}rouzeix--{R}aviart and {M}orley {FEM} with original and
              modified right-hand sides},
   JOURNAL = {Comput. Methods Appl. Math.},
  FJOURNAL = {Computational Methods in Applied Mathematics},
    VOLUME = {21},
      YEAR = {2021},
    NUMBER = {2},
     PAGES = {289--315},
      ISSN = {1609-4840},
   MRCLASS = {65N30 (65N12 65N50)},
  MRNUMBER = {4235819},
       DOI = {10.1515/cmam-2021-0029},
       URL = {https://doi.org/10.1515/cmam-2021-0029},
}

@article {VeigaCanutoNochettoVaccaVerani2023,
    AUTHOR = {Beir\~{a}o da Veiga, L. and Canuto, C. and Nochetto, R. H. and
              Vacca, G. and Verani, M.},
     TITLE = {Adaptive {VEM}: stabilization-free a posteriori error analysis
              and contraction property},
   JOURNAL = {SIAM J. Numer. Anal.},
  FJOURNAL = {SIAM Journal on Numerical Analysis},
    VOLUME = {61},
      YEAR = {2023},
    NUMBER = {2},
     PAGES = {457--494},
      ISSN = {0036-1429,1095-7170},
   MRCLASS = {65N30 (65N50)},
  MRNUMBER = {4556806},
       DOI = {10.1137/21M1458740},
       URL = {https://doi.org/10.1137/21M1458740},
}

@article {CarstensenEigelGedicke2011,
    AUTHOR = {Carstensen, C. and Eigel, M. and Gedicke, J.},
     TITLE = {Computational competition of symmetric mixed {FEM} in linear
              elasticity},
   JOURNAL = {Comput. Methods Appl. Mech. Engrg.},
  FJOURNAL = {Computer Methods in Applied Mechanics and Engineering},
    VOLUME = {200},
      YEAR = {2011},
    NUMBER = {41-44},
     PAGES = {2903--2915},
      ISSN = {0045-7825,1879-2138},
   MRCLASS = {74S05 (65N30 74B05)},
  MRNUMBER = {2824163},
       DOI = {10.1016/j.cma.2011.05.013},
       URL = {https://doi.org/10.1016/j.cma.2011.05.013},
}

@article {BottiDiPietroGuglielmana2019,
    AUTHOR = {Botti, Michele and Di Pietro, Daniele A. and Guglielmana,
              Alessandra},
     TITLE = {A low-order nonconforming method for linear elasticity on
              general meshes},
   JOURNAL = {Comput. Methods Appl. Mech. Engrg.},
  FJOURNAL = {Computer Methods in Applied Mechanics and Engineering},
    VOLUME = {354},
      YEAR = {2019},
     PAGES = {96--118},
      ISSN = {0045-7825,1879-2138},
   MRCLASS = {74B05 (65N30)},
  MRNUMBER = {3958161},
       DOI = {10.1016/j.cma.2019.05.031},
       URL = {https://doi.org/10.1016/j.cma.2019.05.031},
}

@article {KouhiaStenberg1995,
    AUTHOR = {Kouhia, Reijo and Stenberg, Rolf},
     TITLE = {A linear nonconforming finite element method for nearly
              incompressible elasticity and {S}tokes flow},
   JOURNAL = {Comput. Methods Appl. Mech. Engrg.},
  FJOURNAL = {Computer Methods in Applied Mechanics and Engineering},
    VOLUME = {124},
      YEAR = {1995},
    NUMBER = {3},
     PAGES = {195--212},
      ISSN = {0045-7825,1879-2138},
   MRCLASS = {73V05 (76M10)},
  MRNUMBER = {1343077},
       DOI = {10.1016/0045-7825(95)00829-P},
       URL = {https://doi.org/10.1016/0045-7825(95)00829-P},
}

@article {BrennerSung1992,
    AUTHOR = {Brenner, Susanne C. and Sung, Li-Yeng},
     TITLE = {Linear finite element methods for planar linear elasticity},
   JOURNAL = {Math. Comp.},
  FJOURNAL = {Mathematics of Computation},
    VOLUME = {59},
      YEAR = {1992},
    NUMBER = {200},
     PAGES = {321--338},
      ISSN = {0025-5718,1088-6842},
   MRCLASS = {73V05 (65N12 65N30)},
  MRNUMBER = {1140646},
MRREVIEWER = {N.\ Gass},
       DOI = {10.2307/2153060},
       URL = {https://doi.org/10.2307/2153060},
}

@Article{Roessle2000,
	author={R{\"o}ssle, Andreas},
	title={Corner singularities and regularity of weak solutions for the two-dimensional Lam{\'e} equations on domains with angular corners},
	journal={J. Elasticity},
	year={2000},
	volume={60},
	pages={57-75},
}

@article{CarstensenHeuer2024,
      title={A fractional-order trace-dev-div inequality}, 
      author={Carstensen, Carsten and Heuer, Norbert},
      year={2024},
	  journal={arXiv:2403.01291},
}

@phdthesis{Lehrenfeld2010,
  title        = {Hybrid Discontinuous Galerkin methods for solving incompressible flow problems},
  author       = {C. Lehrenfeld},
  year         = {2010},
  school       = {Rheinisch-Westfälischen Technischen Hochschule Aachen},
  type         = {PhD thesis}
}

@article {ArnoldWinther2002,
    AUTHOR = {Arnold, Douglas N. and Winther, Ragnar},
     TITLE = {Mixed finite elements for elasticity},
   JOURNAL = {Numer. Math.},
  FJOURNAL = {Numerische Mathematik},
    VOLUME = {92},
      YEAR = {2002},
    NUMBER = {3},
     PAGES = {401--419},
      ISSN = {0029-599X,0945-3245},
   MRCLASS = {65N30 (65N12 65N15 74G15 74S05)},
  MRNUMBER = {1930384},
MRREVIEWER = {Carsten\ Carstensen},
       DOI = {10.1007/s002110100348},
       URL = {https://doi.org/10.1007/s002110100348},
}

@article {LedererStenberg2024,
    AUTHOR = {Lederer, Philip L. and Stenberg, Rolf},
     TITLE = {Analysis of weakly symmetric mixed finite elements for
              elasticity},
   JOURNAL = {Math. Comp.},
  FJOURNAL = {Mathematics of Computation},
    VOLUME = {93},
      YEAR = {2024},
    NUMBER = {346},
     PAGES = {523--550},
      ISSN = {0025-5718,1088-6842},
   MRCLASS = {65N12 (65N30 74B05 74G15 74S05)},
  MRNUMBER = {4678576},
       DOI = {10.1090/mcom/3865},
       URL = {https://doi.org/10.1090/mcom/3865},
}

@article {virtualelements2023,
    AUTHOR = {Beir\~ao da Veiga, L. and Brezzi, Franco and Marini,
              L. Donatella and Russo, Alessandro},
     TITLE = {The virtual element method},
   JOURNAL = {Acta Numer.},
  FJOURNAL = {Acta Numerica},
    VOLUME = {32},
      YEAR = {2023},
     PAGES = {123--202},
      ISSN = {0962-4929,1474-0508},
   MRCLASS = {65N30},
  MRNUMBER = {4586821},
MRREVIEWER = {Feng\ Wang},
       DOI = {10.1017/S0962492922000095},
       URL = {https://doi.org/10.1017/S0962492922000095},
}

@article {KreuzerVeeser2021,
    AUTHOR = {Kreuzer, Christian and Veeser, Andreas},
     TITLE = {Oscillation in a posteriori error estimation},
   JOURNAL = {Numer. Math.},
  FJOURNAL = {Numerische Mathematik},
    VOLUME = {148},
      YEAR = {2021},
    NUMBER = {1},
     PAGES = {43--78},
      ISSN = {0029-599X,0945-3245},
   MRCLASS = {65N30 (65N12 65N15 65N50)},
  MRNUMBER = {4265898},
MRREVIEWER = {R\"udiger\ Verf\"urth},
       DOI = {10.1007/s00211-021-01194-8},
       URL = {https://doi.org/10.1007/s00211-021-01194-8},
}

@article {CarstensenSchedensack2015,
    AUTHOR = {Carstensen, C. and Schedensack, M.},
     TITLE = {Medius analysis and comparison results for first-order finite
              element methods in linear elasticity},
   JOURNAL = {IMA J. Numer. Anal.},
  FJOURNAL = {IMA Journal of Numerical Analysis},
    VOLUME = {35},
      YEAR = {2015},
    NUMBER = {4},
     PAGES = {1591--1621},
      ISSN = {0272-4979,1464-3642},
   MRCLASS = {65N30 (65N15 74B05)},
  MRNUMBER = {3407237},
MRREVIEWER = {Eduard-Marius\ Craciun},
       DOI = {10.1093/imanum/dru048},
       URL = {https://doi.org/10.1093/imanum/dru048},
}

@article {CarstensenGallistlSchedensack2016,
    AUTHOR = {Carstensen, C. and Gallistl, D. and Schedensack, M.},
     TITLE = {{$L^2$} best approximation of the elastic stress in the
              {A}rnold-{W}inther {FEM}},
   JOURNAL = {IMA J. Numer. Anal.},
  FJOURNAL = {IMA Journal of Numerical Analysis},
    VOLUME = {36},
      YEAR = {2016},
    NUMBER = {3},
     PAGES = {1096--1119},
      ISSN = {0272-4979,1464-3642},
   MRCLASS = {65N30 (65N12 74B05)},
  MRNUMBER = {3523577},
       DOI = {10.1093/imanum/drv051},
       URL = {https://doi.org/10.1093/imanum/drv051},
}

@article {BirdCoombsGiani2019,
    AUTHOR = {Bird, Robert E. and Coombs, William M. and Giani, Stefano},
     TITLE = {A posteriori discontinuous {G}alerkin error estimator for
              linear elasticity},
   JOURNAL = {Appl. Math. Comput.},
  FJOURNAL = {Applied Mathematics and Computation},
    VOLUME = {344/345},
      YEAR = {2019},
     PAGES = {78--96},
      ISSN = {0096-3003,1873-5649},
   MRCLASS = {65N30 (35J57 35Q74 65N15 65N50 74B05)},
  MRNUMBER = {3886404},
       DOI = {10.1016/j.amc.2018.08.039},
       URL = {https://doi.org/10.1016/j.amc.2018.08.039},
}

@article {MoraRivera2020,
    AUTHOR = {Mora, David and Rivera, Gonzalo},
     TITLE = {{\it {A} priori} and {\it a posteriori} error estimates for a
              virtual element spectral analysis for the elasticity
              equations},
   JOURNAL = {IMA J. Numer. Anal.},
  FJOURNAL = {IMA Journal of Numerical Analysis},
    VOLUME = {40},
      YEAR = {2020},
    NUMBER = {1},
     PAGES = {322--357},
      ISSN = {0272-4979,1464-3642},
   MRCLASS = {65N25 (65N15 65N30 74B05 74H45)},
  MRNUMBER = {4050542},
MRREVIEWER = {Nicolae\ Pop},
       DOI = {10.1093/imanum/dry063},
       URL = {https://doi.org/10.1093/imanum/dry063},
}

@article {CarstensenGraessleNataraj2024,
    AUTHOR = {Carstensen, Carsten and Gr\"a\ss le, Benedikt and Nataraj,
              Neela},
     TITLE = {Unifying a posteriori error analysis of five piecewise
              quadratic discretisations for the biharmonic equation},
   JOURNAL = {J. Numer. Math.},
  FJOURNAL = {Journal of Numerical Mathematics},
    VOLUME = {32},
      YEAR = {2024},
    NUMBER = {1},
     PAGES = {77--109},
      ISSN = {1570-2820,1569-3953},
   MRCLASS = {65N30 (65N12 65N50)},
  MRNUMBER = {4717206},
       DOI = {10.1515/jnma-2022-0092},
       URL = {https://doi.org/10.1515/jnma-2022-0092},
}

@article {BabuskaSuri1992,
    AUTHOR = {Babu\v{s}ka, Ivo and Suri, Manil},
     TITLE = {Locking effects in the finite element approximation of
              elasticity problems},
   JOURNAL = {Numer. Math.},
  FJOURNAL = {Numerische Mathematik},
    VOLUME = {62},
      YEAR = {1992},
    NUMBER = {4},
     PAGES = {439--463},
      ISSN = {0029-599X,0945-3245},
   MRCLASS = {65N30 (73V05)},
  MRNUMBER = {1174468},
MRREVIEWER = {Luis\ Trabucho},
       DOI = {10.1007/BF01396238},
       URL = {https://doi.org/10.1007/BF01396238},
}

@article {DiPietroErnLinkeSchienk2016,
    AUTHOR = {Di Pietro, Daniele A. and Ern, Alexandre and Linke, Alexander
              and Schieweck, Friedhelm},
     TITLE = {A discontinuous skeletal method for the viscosity-dependent
              {S}tokes problem},
   JOURNAL = {Comput. Methods Appl. Mech. Engrg.},
  FJOURNAL = {Computer Methods in Applied Mechanics and Engineering},
    VOLUME = {306},
      YEAR = {2016},
     PAGES = {175--195},
      ISSN = {0045-7825,1879-2138},
   MRCLASS = {65N30 (65N12 76D07 76M10)},
  MRNUMBER = {3502564},
       DOI = {10.1016/j.cma.2016.03.033},
       URL = {https://doi.org/10.1016/j.cma.2016.03.033},
}

\end{document}